\documentclass{article}
\usepackage[utf8]{inputenc}
\usepackage{bm,color}
\usepackage{bbm}
\usepackage{a4wide}
\usepackage{graphicx}
\usepackage{amsmath,amssymb,amsfonts,anyfontsize,amsthm}
\usepackage{calligra,epsfig}
\usepackage{indentfirst,enumerate,mathrsfs}
\usepackage{array,amscd}
\usepackage{multirow}

\newcommand{\da}{\mathcal{D}(A)}
\newcommand{\dl}{\mathcal{D}(L)}

\newcommand{\DD}{\mathcal D}
\newcommand{\EE}{\mathbb E}
\newcommand{\NN}{\mathbb N}
\newcommand{\PP}{\mathbb P}
\newcommand{\RR}{\mathbb R}
\newcommand{\Ro}{R_0}
\newcommand{\HH}{\mathcal H}
\newcommand{\Ho}{\mathcal H}
\newcommand{\Ht}{\mathcal H'}
\newcommand{\xx}{\mathbf{x}}
\newcommand{\yy}{\mathbf{y}}
\newcommand{\zz}{\mathbf{z}}
\newcommand{\NT}{\mathcal{N}_{T_\nu}}

\newcommand{\al}{\alpha}
\newcommand{\la}{\lambda}
\newcommand{\ga}{g_\lambda}
\newcommand{\ra}{r_\lambda}

\newcommand{\fz}{f_{\mathbf{z},\la}}
\newcommand{\uz}{u_{\mathbf{z},\la}}
\newcommand{\fp}{f^\dagger} %{\hat{f}}
\newcommand{\up}{u^\dagger}

\newcommand{\gp}{g^\dagger}%{\hat{g}}{g_\rho}%
\newcommand{\fbar}{\bar f}
\newcommand{\fhat}{f_{{\bf z}}}
\newcommand{\LL}{\mathscr{L}^2(X,\nu)}
\newcommand{\LLL}{\mathscr{L}}

\newcommand{\sx}{S_\xx}

\newcommand{\bx}{B_\xx}
\newcommand{\tx}{T_\xx}
\newcommand{\ip}{I_\nu}
\newcommand{\lp}{L_\nu}
\newcommand{\bp}{B_\nu}
\newcommand{\tp}{T_\nu}
\newcommand{\op}[1]{\operatorname{#1}}

\newcommand{\range}{\mathcal R}
\newcommand{\tsup}{\sup\limits_{t\in[0,\kappa^2]}}
\newcommand{\argmin}{\operatornamewithlimits{argmin}}
\newcommand{\paren}[1]{\left(#1\right)}
\newcommand{\brac}[1]{\left\{#1\right\}}
\newcommand{\sbrac}[1]{\left[#1\right]}
\newcommand{\norm}[1]{\left\|{#1}\right\|}
\newcommand{\scalar}[3]{\langle{ #1},{#2} \rangle_{#3}}
\newcommand{\inner}[1]{\left\langle#1\right\rangle}
\newcommand{\abs}[1]{\left\lvert #1 \right\rvert}
\RequirePackage{color}

\newcommand{\tapio}[1]{{\color{red} #1}}

\newtheorem{theorem}{Theorem}[section]

\newtheorem{corollary}[theorem]{Corollary}
\newtheorem{proposition}[theorem]{Proposition}
\theoremstyle{definition}
\newtheorem{definition}[theorem]{Definition}
\newtheorem{assumption}{Assumption}
\theoremstyle{remark}
\newtheorem{remark}[theorem]{Remark}
\newtheorem{example}[theorem]{Example}

\numberwithin{equation}{section}

\input{macros_nicole}

\title{Statistical inverse learning problems with random observations}
\author{Abhishake, Tapio Helin, Nicole M\"ucke \footnote{corresponding author nicole.muecke@tu-braunschweig.de}}
\date{\today}

\begin{document}

\maketitle
\begin{abstract}
We provide an overview of recent progress in statistical inverse problems with random experimental design, 
covering both linear and nonlinear inverse problems. Different regularization schemes have been studied to 
produce robust and stable solutions. We discuss recent results in spectral regularization methods and regularization 
by projection, exploring both approaches within the context of Hilbert scales and presenting new insights particularly 
in regularization by projection. Additionally, we overview recent advancements in regularization using convex penalties. 
Convergence rates are analyzed in terms of the sample size in a probabilistic sense, yielding minimax rates in both 
expectation and probability. To achieve these results, the structure of reproducing kernel Hilbert spaces is leveraged 
to establish minimax rates in the statistical learning setting. We detail the assumptions underpinning these key elements 
of our proofs. Finally, we demonstrate the application of these concepts to nonlinear inverse problems in 
pharmacokinetic/pharmacodynamic (PK/PD) models, where the task is to predict changes in drug concentrations in patients.
\end{abstract}

\tableofcontents

\section{Introduction}\label{Sec:Intro}
Statistical inverse learning refers to the process of inferring an unknown function~$\fp$ by using statistical methods on point evaluations of a related function $\gp$. These evaluations might be sparse or noisy, and the connection between $\fp$ and $\gp$ is established through an ill-posed mathematical model   
\begin{equation}\label{Model}
  A(\fp) = \gp,\qquad\text{for}\quad \fp\in \Ho \quad \text{and} \quad \gp\in\Ht,
\end{equation}
where $A$ can be a linear or nonlinear operator between a real separable Hilbert spaces~$\Ho$ and a reproducing kernel Hilber space $\Ht$. In particular, $\Ht$ represents the set of functions $g: X \to \RR$, where $X$ is a standard Borel space (the input space) and $\RR$ is a space of real numbers (the output space). The ill-posedness of~\eqref{Model} means the problem lacks a unique solution or is sensitive to data perturbations such as measurement noise.

In this work, we focus on random design regression in inverse problems, where the data~$\zz = \paren{(x_i,y_i)}_{i=1}^m$ emerges due to random observations
\begin{equation}\label{eq:model}
y_{i} = \gp(x_{i}) + \varepsilon_{i},\quad i=1,\dots,m,
\end{equation}
where~$\varepsilon_i$ is the observational noise, and~$m$ denotes the sample size. Here, we assume that the random observations~$(x_i,y_i)$ are i.i.d. and follow the \emph{unknown} probability distribution~$\rho$, defined on the sample space~$Z=X\times \RR$. The key objective is to understand interplay between uncertainty of the design measure and optimal reconstruction rates.

Classical applications in inverse problems community often involve laboratory settings, where deterministic design of observational points $\paren{x_i}_{i=1}^m$ is possible. However, there is a growing literature considering inverse problems where observational design contains inherent uncertainties related to areas such as machine learning~\cite{Blanchard18,Hartung21,Rastogi20} and imaging~\cite{nickl2023bayesian, nickl2020convergence, burger2017variational}. Statistical inverse learning provides a principled approach to analyse concentration in such context.

The starting point in the inverse learning paradigm with random design is that we do not necessarily possess exact information regarding the distribution $\rho$. In this regard, our work differs from other avenues of random design regression in inverse problems such as~\cite{giordano2020consistency, nickl2020convergence, monard2021consistent}, where a specific (uniform) random design is assumed. Statistical inverse learning aims to assess how assumptions on $\rho$ influence concentration of given statistical estimators.
That being said, it is known that under suitable regularity assumptions on $\rho$, the observation scheme in~\eqref{eq:model} is asymptotically equivalent (in a Le Cam sense of statistical experiments) to the full observation of $\gp$ corrupted by Gaussian noise with vanishing variance~\cite{reiss2008asymptotic}. 

The ill-posedness of inverse problems is usually mitigated by regularization techniques \cite{Hanke95} which stabilize the reconstruction and prevent overfitting. These techniques introduce additional information or constraints to the problem to guide the solution towards a more reasonable and stable outcome. In this work, we overview recent concentration results for penalized least squares estimators \begin{equation}
    \fz = \argmin_{f\in\da\subset\Ho} J_\la(f), \quad J_\la(f) = \frac 1m \sum_{i=1}^m \paren{A(f)(x_i)-y_i}^2 + \la G(f),
\end{equation}
where the first term in $J_\la$ is called the empirical error, the second term is the penalty term $G: \Ho \to \RR_+$ and $\la>0$ is the regularization parameter that balances both terms. From the recent literature studying statistical inverse learning with random design, we overview three different regularization strategies: classical Tikhonov-type penalties, projection methods and general convex penalties. We separately consider results appearing for linear and non-linear forward mapping $A$.

Concentration results in inverse learning require assumptions limiting the complexity of the set of probability models that may produce the data $\zz$.
Such assumptions are typically characterized by three conditions: First, a \emph{source condition} limiting complexity of the source set (set of possible solutions $f^\dagger$), second, mapping properties of the forward operator, e.g. \emph{smoothness condition} for linear operators and, third, \emph{design measure} $\nu$ generating points $(x_i)_{i=1}^m \subset X^m$.
As we will see later, these conditions are often intertwined and the analysis of different penalties has been coupled with varying assumptions on these conditions. Here, with some cost of generality, we try to unify the exposition for clarity.

\subsection{Literature overview}

The literature on statistical learning, reproducing kernel Hilbert spaces and kernel regression, and regularization theory is vast and beyond the scope of this treatise. We guide interested readers to \cite{Aronszajn50,Cucker02,Cucker07,Engl96,Steinwart08,Vapnik98} for classical sources on these subjects.

Connections of kernel regression methods to regularization theory were first investigated in \cite{de2006discretization, DeVito05, gerfo2008spectral}. Early work on upper rates of convergence in a reproducing kernel Hilbert space was carried out in \cite{cucker2002best} utilizing a covering number technique. Subsequent work includes \cite{smale2005shannon, smale2007learning, Bauer07, yao2007early, Caponnetto07} culminating to the minimax optimal rates demonstrated in \cite{Blanchard18}. 
In this paper, we focus on results on random design regression in inverse problems introduced in \cite{Blanchard18,Rastogi23} and extended to non-linear problems~\cite{Rastogi20,Rastogi20a,Rastogi22}, and projection and convex regularization strategies in \cite{Bubba21, Helin22}.
Here, let us note that convex penalties of type $G(f) = \frac 1p \norm{f}_{\Ho}^p$ were considered in statistical learning in \cite{mendelson2010, steinwart2009optimal}. However, these results are not directly comparable to the inverse problem literature as, among other things, convergence takes place on a different domain $\Ho \neq \Ho'$.

The literature around statistical inverse problems has also expanded rapidly during last two decades. For a general overview, we refer to \cite{kaipio2007statistical, stuart2010inverse, cavalier2011inverse, nickl2023bayesian}. Early works involving spectral regularization methods include \cite{bissantz2007convergence, guastavino2020convergence, Bauer07, yao2007early, de2005learning}. Projection methods are also well-studied and their minimax optimality is understood in various settings with deterministic design \cite{wahba1977practical, nychka1989convergence, johnstone1991discretization, donoho1995nonlinear, mair1996statistical, lukas1998comparisons, cavalier2002sharp,chow1999statistical}.
In the context of convex regularization strategies, 
Bregman distances were first introduced in \cite{burgerosher} and have been frequently used since as an error measure for studying convergence rates in Banach spaces. Statistical inverse problems with convex penalties have been studied in \cite{burger2018large, weidling2020optimal}. For a recent discussion on Bregman distances we refer to \cite{benning2018modern}. 
Furthermore, adaptivity and parameter choice rules are studied in 
\cite{caponnetto2010cross, de2010adaptive, lu2020balancing, blanchard2019lepskii, blanchard2018early, blanchard2018optimal,  stankewitz2020smoothed, alberti2021learning}.

For other related works involving non-Gaussian likelihood models, nonlinear forward mapping or iterative methods, see e.g. \cite{hohage2016inverse, hohage2008nonlinear, harrach2020beyond, jin2018regularizing, blanchard2018optimal, clason2012semismooth, blanchard2012discrepancy}.

This paper is organized as follows. In Section \ref{Sec:math.prelim} we describe our problem setting in detail and cover mathematical background to the field including Hilbert scales and model class assumptions. In addition, we discuss minimax optimality of statistical estimators. In Section \ref{sec:linear} we overview the results on spectral regularization (in the classical framework and in Hilbert scales), projection methods and convex penalties. Moreover, literature on parameter choice rules with random design is discussed. In Section \ref{sec:non-linear} we describe non-linear inverse problems isolating different scenarios on differentiability or stability on the forward map in classical setup or in Hilbert scales. Finally, we discuss an application in dynamic system monitoring in Section \ref{Sec:Applicaions} and give a brief outlook in Section \ref{Sec:outlook}.

%%%%%%%%%%%%%%%%%%%%%%%%%%%%%%%%%%%%%%%%%%%%%%%%%%%%%%%%%%%%%%%%
%%%%%%%%%%% Mathematical preliminaries %%%%%%%%%%%%%%%%%%%%%%%%%

\section{Statistical Inverse Learning Theory}\label{Sec:math.prelim}

In this section, we present fundamental concepts in supervised inverse statistical learning utilizing the least squares method \cite{Blanchard18,Rastogi20,Helin22}.

\subsection{Problem setting}\label{Sec:Problem.setting}

\noindent
{\bf Statistical Model and Noise Assumption.} 
In the context of supervised inverse learning, the objective is to learn a function from random samples of input-output pairs $(x, y) \in X \times Y$. Here, we assume that the input space $X$ is a Polish space, such as $\mathbb{R}^d$, and the output space is denoted as $Y = \mathbb{R}$. The product space $Z = X \times \mathbb{R}$ is equipped with a probability measure $\rho$, defined on the Borel $\sigma$-algebra of $X \times \mathbb{R}$.

The probability distribution $\rho$ is only known through a training set $\bz = \paren{(x_1, y_1), \ldots, (x_m, y_m)} \in Z^m$, consisting of pairs sampled independently and identically according to $\rho$. Given this training set $\bz$, the goal in supervised inverse learning is to find an estimate $\fhat: X \to \mathbb{R}$ so that, for a new unlabeled input $x \in X$, the value $A (\fhat)(x)$ serves as a reliable approximation of the true label $y$.
%while following the model \eqref{eq:model}.

To formalize this objective, we introduce the notion of expected square loss for any measurable function $f: X \to \mathbb{R}$:
\[
\cE(f) = \int_{X \times \mathbb{R}} (A(f)(x) - y)^2 \, d\rho(x, y). \label{eq:expected-risk} \tag{2}
\]
This loss function captures the discrepancy between the predicted values and the true labels, with pairs $(x, y)$ that are more likely to be sampled exerting a greater contribution to the overall error. 
The expected loss can be expressed in an alternative form, assuming that
\[
\int_{X \times \mathbb{R}} y^2 \, d\rho(x, y) < \infty. \label{eq:2} \tag{3}
\]

We recall a useful integral decomposition. Let $h: X \times \mathbb{R} \to \mathbb{R}$ be a measurable function. Then, the integral of $h$ over $X \times \mathbb{R}$ can be decomposed as
\[
\int_{X \times \mathbb{R}} h(x, y) \, d\rho(x, y) =
\int_X \int_{\mathbb{R}} h(x, y) \, d\rho(y | x)\, d\nu(x), \label{eq:meas-dis} \tag{4}
\]
where $\nu$ represents the marginal measure on $X$, and $\rho(\cdot | x)$ is the conditional probability measure on $\mathbb{R}$ given $x \in X$. Recall that the existence of $\rho(\cdot | x)$ is ensured by our assumption on $Z$, see e.g. \cite{shao2003mathematical}.
%First of all, we discuss the unknown probability measure~$\rho$ and the corresponding target solution. Suppose the probability measure $\rho$ can be decomposed as follows:
%$$\rho(x, y) = \rho(y|x)\nu(x),$$
%where $\rho(y|x)$ represents the conditional probability distribution of $y$ %given $x$, and $\nu$ is the marginal probability distribution over $X$.
%In the context of the statistical inverse problem, the goodness of an estimator $f$ can be assessed using the expected risk: \nic{No, it's the reconstruction error that gives a measure for the quality!}
%$$\mathcal{E}(f)=\int_{Z}\paren{A(f)(x)-y}^2d\rho(x,y)$$
%Moreover, we make the assumption that $\int_{\RR} y^2 d\rho(y|x)<\infty$ for all %$x\in X$. 
Denoting by  $$g_\rho(x)=\int_\RR y \, d\rho(y|x),$$ the conditional mean function, the expected risk can be written as
%$$\mathcal{E}(f)=\int_{X}\paren{A(f)(x)-g_\rho(x)}^2d\nu(x)+\int_{Z}\paren{g_\rho(x)-y}^2d\rho(x,y).$$
%For the norm~$\norm{g}_\nu^2=\int_X \paren{g(x)}^2 d\nu(x)$, this yields
\[ \mathcal{E}(f)=\norm{A(f)-g_\rho}^2_\nu  + \cE(g_\rho) ,\]
where we denote the $L^2$-norm on $X$ by $\norm{g}_\nu^2=\int_X |g(x)|^2\, d\nu(x)$. 
Therefore, finding the minimizer of the expected risk is equivalent to identifying the minimizer of the first term on the right side in the above equation. In other words, the minimizers of expected risk typically coincide with the classical minimum norm solutions of inverse problems with the difference that data fidelity is weighted by the design measure. Moreover, unlike in statistical learning we are interested in providing convergence rates for the unknown $f$ assuming some ground truth.

\vspace{0.2cm}

\begin{remark}
When there exists an element~$f_\rho \in \da~$ such that $g_\rho = A(f_\rho)$, then $f_\rho$ is a minimizer of the expected risk. In the scenario of centered noise  $\int_\RR \varepsilon d\rho(y|x)=0$, we obtain from the equations~\eqref{Model},~\eqref{eq:model} that $\int_\RR y d\rho(y|x)= \gp(x) = A(\fp)(x)$. Consequently, in this case for the injective operator~$A$, the minimizer of the expected risk, denoted as $f_\rho$, is equal to $\fp$.
\end{remark}

%\cite{Blanchard18} considered the following assumption for the true solution. 

In what follows, we state our assumptions on the true solution of the inverse problem \eqref{Model} and the noise in \eqref{eq:model}, see also \cite{Blanchard18}.

\begin{assumption}[True solution]
\label{Ass:fp}
The conditional expectation of~$y$ given~$x$, with respect to the probability distribution~$\rho$, is guaranteed to exist almost surely (a.s.). Furthermore, there exists an element~$\fp \in \da~$ such that:
  \begin{equation*}
\int_\RR y \, d\rho(y|x)= \gp(x) = [A(\fp)](x), \text{ for all } x\in X.
  \end{equation*}
\end{assumption} 

%The element~$\fp$ serves as the true solution that we aim to estimate.

The noise satisfies the following Bernstein condition:
\begin{assumption}[Noise condition]
\label{Ass:noise}
There exist positive constants~$\Sigma > 0$ and $M > 0$ such that for any integer $l \geq 2$  
\begin{equation*}
%\int_\RR\left(e^{\abs{\varepsilon}/M}-\frac{\abs{\varepsilon}}{M}-1\right)d\rho(y|x)\leq\frac{\Sigma^2}{2M^2}, %\qquad \varepsilon=y-A(\fp)(x).
\int_\RR\abs{y-[A(\fp)](x)}^l d\rho(y|x) \leq \frac{1}{2} l! \Sigma^2 M^{l-2},
\end{equation*}
$\nu$-almost surely. 
\end{assumption}

%The above condition implies that for any integer $m \geq 2$:
%$$\int_\RR\abs{\varepsilon}^m d\rho(y|x) \leq \frac{1}{2} m! \Sigma^2 M^{m-2}.$$  
Assumption~\ref{Ass:noise} characterizes the nature of the noise affecting the output~$y$. This equation remains valid under various scenarios, including cases where the noise~$\varepsilon$ is bounded and follows either a Gaussian distribution or a sub-Gaussian distribution with zero mean and it is independent from $x$~\cite{van96}.

\vspace{0.2cm}
\noindent
{\bf RKHS Structure.} We endow the space $\cH'$ with an additional structure that allows us to exploit the rich theory of reproducing kernels.

\begin{assumption}[RKHS Structure] \label{Ass:Kernel} 
The space~$\Ht$ is supposed to be a \emph{reproducing kernel Hilbert space} (RKHS) of functions~$g:X \to \RR$ associated to a measurable bounded  kernel~$K:X \times X\to \RR$ with $\kappa^2:=\sup\limits_{x \in X} K(x,x) <\infty.$ We additionally assume that $\cH'$ can be continuously embedded into $\LL$. 
We denote the inclusion by $I_\nu: \cH' \hookrightarrow \LL$. 
\end{assumption}

We recall from~\cite{Steinwart08}, Chapter 4, that this assumption is equivalent to assuming that for any $x \in X$,  the evaluation functional 
\begin{align}\label{eq:eval}
    S_x: \cH' &\to \mbr \nonumber \\
    g &\mapsto S_x(g) := g(x) 
\end{align}
is continuous. By Riesz’s representation theorem, there exists a unique $ F'_x\in\cH'$ 
such that $g(x)=\inner{g, F'_x}_{\cH'}$ and 
\[ K(x , x') = \inner{F'_x , F'_{x'}}_{\cH'} \; , \quad x,x' \in X.\]
In particular, for any $g \in \operatorname{Im}(A)$, we may write
\[ g(x)= [A(f)](x) = \inner{A(f), F'_x}_{\cH'} , \]
for some $f \in \cH$ and for all $x \in X$. If $A$ is linear, we even have 
\begin{equation}
\label{eq:feature-map-lin}
 [A(f)](x) = \inner{f, F_x}_{\cH} ,  \quad F_x := A^*  F'_x \;. 
\end{equation}
From standard theory, we know that the feature map $x \mapsto F_x$ gives rise to another 
reproducing kernel
\[ K_A (x, x') := \inner{ F_x , F_{x'}}_{\cH} , \quad x, x' \in X, \]
depending on the operator $A$. The associated RKHS $\cH_{K_A}$ coincides with the image of $\cH$ under $A$:  
$\cH_{K_A} = \operatorname{Im}(A)$ and $A$ acts as a partial isometry, see~\cite[Proposition 2.2]{Blanchard18}, for an extended discussion. 
Hence, in the linear case, we can forget about any structure on the space $\cH'$ and consider $A$ as an operator from $\cH$ 
to $\cH_{K_A}$ by assuming continuity of $S_x$ on $\operatorname{Im}(A)$ only, for any $x \in X$. We formulate 
this as an independent assumption and will analyze this case in Section~\ref{sec:linear} more deeply. 

\begin{assumption}[Linear Case] \label{Ass:kernel.linear}
The operator $A: \cH \to \cH':=\operatorname{Im}(A)$ is linear. Moreover, the evaluation functional 
$S_x: \cH' \to \mbr $ defined by~\eqref{eq:eval} is continuous for all $x \in X$ and the 
map $x \mapsto S_x g $ is measurable  for all $g\in\cH'$. 
\end{assumption}

We emphasize once more that in the general non-linear case in Assumption~\ref{Ass:Kernel}, the kernel $K$ and thus the RKHS do not depend on the 
operator $A$, while in the linear case, the operator $A$ is explicitly used to construct the kernel $K_A$.  

\vspace{0.2cm}

\begin{remark}[Equivalence with Classical Kernel Learning Setting]
Under Assumption~\ref{Ass:kernel.linear}, the inverse problem~\eqref{Model} is naturally equivalent to the classical  kernel learning problem (the \emph{direct} problem) in non-parametric regression in 
\cite{Bauer07, Caponnetto07, DeVito05}, that follows the model 
\[  y_i =  {\gp}(x_i) + \eps_i . \]
The goal is to estimate the function $\gp$. Kernel methods \emph{posit} that $\gp$ belongs to some reproducing
kernel Hilbert space %$\cH_K$ with kernel $K$ 
and construct an estimate $\hat g $  of $\gp$ based
on the observed data. The reconstruction error $\hat g - \gp$ is  analyzed in $L^2(\nu)$-norm
or in RKHS-norm. In the setting of Assumption~\ref{Ass:kernel.linear} and given the model~\eqref{Model}, the linear operator $A$ is a partial isometry and hence, defining $\hat f := A^+ \hat g$, $f^\dagger:= A^+ \gp$ with pseudoinverse $A^+$ gives 
\[ \norm{f^\dagger - \hat f}_{\cH} = \norm{\gp - \hat g}_{\cH'} , \]
see~\cite{Blanchard18}. 
%\nic{ clash of notation, in statistics, hats as in $\hat f$ are reserved for estimators} 
Hence, a bound established for the direct learning
setting in the sense of the RKHS-norm reconstruction also applies to the
inverse problem reconstruction error. 
Furthermore, the eigenvalue decay conditions and the source conditions involving the operator $\tp$ introduced in Section~\ref{sec:covariance-ops} are, via the same isometry, equivalent to similar conditions involving the kernel integral operator in the direct learning setting, as considered, for instance, in~\cite{Bauer07, Caponnetto07, DeVito05}. It follows that estimates in $\cH'$-norm 
available from those references are directly applicable to the inverse learning setting. 
\end{remark}

\vspace{0.2cm}

\begin{remark}
When employing kernel methods in statistical inverse learning problem~\eqref{Model}, one tackles the ensuing minimization problem:
$$\inf\limits_{f\in \Ho}\int_{X\times \RR}\paren{(Af)(x)-y}^2d\rho(x,y).$$

Given Assumption~\ref{Ass:kernel.linear} for~$(Af)(x)=\inner{f,F_x}_{\Ho}$, the previous expression can be reformulated as:
$$\inf\limits_{f\in \Ho}\int_{X\times \RR}\paren{\inner{f,F_x}_{\Ho}-y}^2d\rho(x,y).$$

By introducing an alternative probability measure~$\rho'(F_x,y)=\rho(x,y)$, the above simplifies to the form:
$$\inf\limits_{f\in \Ho}\int_{X\times \RR}\paren{\inner{f,F_x}_{\Ho}-y}^2d\rho'(F_x,y)=\inf\limits_{f\in \Ho}\int_{\Ho \times \RR}\paren{\inner{f,\xi}_{\Ho}-y}^2d\rho'(\xi,y).$$

Therefore, minimizing the considered statistical inverse problem~\eqref{eq:model} in the linear case is equivalent to solving the following problem:
\begin{equation}\label{lin.ip.form}
    y_i=\inner{f,\xi_i}_\Ho+\varepsilon_i, \quad i=1,\ldots,m.
\end{equation}
\end{remark}

%%%%%%%%%%%%%%%%%%%%%%%%%%%%%%%%%%%%%%%%%%%%%%%%%%%%%%%%%%
%%%%%%%%%%%% Hilbert Scales
%%%%%%%%%%%%%%%%%%%%%%%%%%%%%%%%%%%%%%%%%%%%%%%%%%%%%%%%%%

%\vspace{0.2cm}
\noindent
\subsection{Hilbert Scales}\label{Sec:Hilbert.Scales}

In the context of regularized statistical inverse problems, the utilization of a Hilbert scale enables a delicate balance between the fidelity of the estimate to the observed data and the smoothness of the estimated function~\cite{Natterer84,Engl96}. 
Hilbert scales, comprised of a nested sequence of Hilbert spaces, facilitate systematic control over the smoothness and complexity of functions. Each space in the sequence is compactly embedded in the next, offering a structured approach to managing these properties. 
%One may alter the norm in $\cH$ for the regularization process in order to enforce special features of the regularized approximation, such as smoothness. 
Pioneering works by Tikhonov  \cite{tikhonov1963regularization, tikhonov1963solution}, extended the regularization process to a more general optimization problem, minimizing a functional of the form 
\[  \norm{A(f) - \gp }^2_\nu + \lam \norm{L f}_{\cH}^2 \,, \quad f \in \cD(L)\;,  \]
where $L: \cD(L) \to \cH$ is an unbounded operator with domain $\cD(L)$. A notable instance occurs when $L$ is a differential operator, effectively enforcing smoothness in the approximation, given that the search space is a subset of $\cD(L)$. 

Remarkably, convergence of these regularized solutions can be achieved in the original $\cH$-norm even when no solution of \eqref{Model} is an element of $\cD(L)$. In this case, we are in the setting of \emph{oversmoothing}. We now give a rigorous definition:

\begin{definition}[Hilbert Scale]
\label{def:hilbertscale}
Let $a \in \mbr$ and let $L: \cD(L) \to \cH$ be a  densely defined unbounded self-adjoint positive operator, with $\cD(L)\subset \cH$. We set  
$\tilde \cH := \bigcap_{a \in \mbr}\cD(L^a)$ and define on $\tilde \cH$ the inner product and norm 
\[ \inner{f_1 , f_2}_{a}:= \inner{L^af_1 , L^af_2}_{\cH}\;, \quad \norm{f}_a:=\norm{L^af}_\cH \;. \]
The Hilbert space $\cH_a$ is defined as the completion of $\tilde \cH$ with respect to the norm $\norm{\cdot}_a$. The scale $(\cH_a)_{a \in \mbr}$ is called the \emph{Hilbert scale induced by $L$}. 
\end{definition}
Note that by spectral theory, the powers $L^a$ are defined on $\tilde \cH$. Moreover, for any $-\infty < a < a' < \infty$, we have $\cH_{a'} \subset \cH_a$, where the embedding is dense and continuous, see \cite[Section 8.4]{Engl96} for mathematical details. 
%We define the Hilbert Scales in terms of a densely defined unbounded self-adjoint positive operator denoted as~$L$, with its domain and range residing in the space~$\Ho$. According to spectral theory, the operator~$L^a: \DD(L^a) \to \HH$  is well-defined for~$a \in \RR$, and the space~$\HH_a := \DD(L^a), a \geq 0$, equipped with the inner product~$\scalar{f}{g}{\HH_a}=\scalar{ L^a f}{L^a g}{\HH},~ f, g \in\HH_a$ are Hilbert spaces. For~$a < 0$, the space~$\HH_a$ is defined as the completion of $\Ho$ under the norm~$\norm{f}_a := \scalar{f}{f}{a}^{1/2}$. This collection of Hilbert spaces is referred to as the Hilbert scale induced by~$L$.

In many instances, it is assumed that the scale of Hilbert spaces corresponds to a scale of Sobolev spaces, and the smoothing properties of the underlying operator $A$ are assessed in relation to this scale. Additionally, the smoothness of the solution is characterized by assuming it belongs to a specific space within this scale. 
Additional examples of Hilbert Scales pertinent to learning in RKHSs can be found in the work of \cite{Mucke20}.

There has been growing interest in regularization within Hilbert scales for statistical inverse problems, particularly in the context of the Bayesian approach~\cite{Gugushvili20,Agapiou22}. Recently, this approach gained attention in statistical learning and statistical inverse learning~\cite{Rastogi20a,Rastogi23,Mucke20}.

%%%%%%%%%%%%%%%%%%%%%%%%%%%%%%%%%%%%%%%%%%%%%%%%%%%%%%%%%%
%%%%%%%%%%%% A Priori Assumptions: Source condition and effective dimension 
%%%%%%%%%%%%%%%%%%%%%%%%%%%%%%%%%%%%%%%%%%%%%%%%%%%%%%%%%%

\subsection{A Priori Assumptions}
\label{sec:covariance-ops}

In classical statistical learning theory for non-parametric regression, the convergence speed of an estimator to the target function can be exceedingly slow as the sample size grows, a phenomenon known as the \emph{No Free Lunch Theorem}. To address this issue,   the set of all data-generating distributions is indirectly limited, for more efficient learning. A priori assumptions play a critical role in this process by enabling the inclusion of prior knowledge and information, ultimately leading to faster rates of convergence.

In what follows, we incorporate prior knowledge in two ways: First, we gauge the \emph{degree of ill-posedness} of the inverse problem by assuming a certain decay behavior of the eigenvalues of a second moment operator associated to $A$. The faster the decay, the more ill-posed the problem is. Secondly, we impose a smoothness/ sparsity  constraint on the target function $\fp$. The smoother the target, the easier the reconstruction is. We now provide a formalization of both concepts utilizing the following covariance operator related to feature maps.

\begin{definition}[Uncentered Covariance Operator]
Let $\Phi: X \to \cH$ be a measurable feature map with
\begin{equation}
\label{eq:trace}
    \mbe_\nu[\norm{\Phi(x)}_\cH^2] < \infty.
\end{equation}

The \emph{uncentered covariance operator} associated to $\Phi$ is defined to be the operator 
$ \tp : \cH \to \cH$ with 
\[ \tp := \mbe_\nu[\Phi(x) \otimes \Phi(x)  ] 
= \int_X \inner{\cdot, \Phi(x)} \Phi(x) \; d\nu(x). \]
\end{definition}
\vspace{0.2cm}

The operator $\tp$ is positive, self-adjoint and compact, in particular trace class by \eqref{eq:trace} with $\sum_{j \in \mbn} \sigma_j < \infty $, where $(\sigma_j)_{j \in \mbn}$ denotes the decreasing sequence of eigenvalues of $\tp$.

Depending on the context, linear or non-linear problems, we will consider different feature maps $\Phi$, giving rise to different operators $\tp$. However, the non-linear approach contains the linear approach as a special case, as exemplified below.

\vspace{0.2cm}

\begin{example}[Linear Case] 
If $A:\cH \to \cH'$ is linear, we let $\Phi(x):=A^*F_x' = F_x$, see~\eqref{eq:feature-map-lin}. Setting $\bp=\ip A$, we also have the alternative representation of the covariance operator $\tp=\bp^*\bp$, see e.g. \cite{Blanchard18, Caponnetto07}. 
\end{example}

\vspace{0.2cm}

\begin{example}[Non-Linear Case]
\label{ex:non-linear}
Assume that $A:\cH \to \cH'$ is Fr\'{e}chet differentiable at $\fp$ with derivative $A'(\fp) :\cH \to \cH'$. We let $\Phi(x) := A'(\fp)^* F'_x$.  Setting  $\bp=\ip A'(\fp)$, we again obtain $\tp:=\bp^*\bp$. Note that we recover the linear case since $A'(\fp) = A$ if $A$ is linear. 
\end{example}

%\vspace{0.2cm}

%\begin{example}[Hilbert Scales] 
%Assume $L:\cD(L) \subset \cH \to \cH$ generates a Hilbert scale. If $A$ is  %Fr\'{e}chet differentiable at $\fp$, we set $\Phi(x):=L^{-1}  A'(\fp)^* F'_x$. Setting, 
%$\bp= \ip A'(\fp)$ gives  $\tp:=\bp^*\bp$ and the covariance operator $T=L^{-1}(L^{-1}\tp L^{-1}) %L^{-1}$, including again the linear case with $A'(\fp) = A$.
%\end{example}

%We obtain the corresponding covariance operators from the equations~\eqref{lin.equ},~\eqref{nonlin.equ},~\eqref{mod.regu.HS} in different cases.  The covariance operator takes the form $\tp
%=\bp^*\bp$ for $\bp=\ip A,~\ip A'(\fp),~\ip AL^{-1},~\ip A'(\fp)L^{-1}$ in the linear, %nonlinear case and linear, nonlinear case in Hilbert scale setting, respectively. 

In our given examples we observe that the second moment operator $\tp$ depends on both, the marginal distribution $\nu$ and the operator $A$ itself, through the feature map $\Phi$. The properties of $\tp$ are crucial for deriving fast rates.

\begin{remark}(Normal Equation)
Let $A$ be a linear operator, and let $f^\dagger$ be the expected risk minimizer, i.e., the least-squares solution of the model \eqref{Model} given by   
\[ f^\dagger = \argmin_{f \in \cH} \norm{B_\nu f - \gp}_\nu = \argmin_{f \in \cH} \cE(f)\;. \]
This minimization is equivalent to the validity of the Gaussian normal equation \cite{Engl96}
\begin{equation}\label{lin.equ}
  \tp f^\dagger =\bp^*\gp .
\end{equation} 
\end{remark}

%\nic{Abi old: 
%When $\range(A)$ is contained in an RKHS, we have $A(f)=\ip A(f)$ for the canonical injection map~$\ip:\Ht \to \LL$, consequently,
%$\norm{A(f)-\gp}_\nu=\norm{\ip A(f)-\gp}_\nu$. In the linear case, the %solution~$f_\rho$ of this error term satisfies the linear equation:
%\begin{equation}\label{popu.equ}
%   (\ip A)^*\ip A (f_\rho) =(\ip A)^* \gp
%\end{equation}

%Hence, we see that the covaiance operator for the linear inverse problem \nic{ the %covariance operator is not defined yet, comes later}
%\begin{equation}\label{lin.equ}
%    \ip A (f)= \gp
%\end{equation}
%is $\tp=(\ip A)^*(\ip A)$~\cite{Engl96}.
%}

\vspace{0.2cm}
\noindent
{\bf Eigenvalue Decay.} 
We impose a-priori assumptions on the set of marginal distributions on $X$ by assuming a  decay rate for the eigenvalues of the covariance operator $\tp$, reflecting the severity of the ill-posed nature of \eqref{Model}. 

\begin{assumption}\label{Ass:eigen.decay}
Suppose the eigenvalues $\brac{\sigma_i}_{i \in \NN}$ of the integral operator $\tp$ follow a polynomial decay: For fixed positive constants $\al, \beta$ and $b < 1$, we asume that 
\begin{enumerate}[(i)]
    \item $ \sigma_i \leq \beta i^{-\frac{1}{b}}$, for all $i \in \NN$,
%\begin{equation*}
%   \sigma_i \leq \beta i^{-\frac{1}{b}}, \quad i \in \NN,
%\end{equation*}
\item $ \al i^{-\frac{1}{b}} \leq  \sigma_i$, for all $i \in \NN$, 
%\begin{equation*}
%   \al i^{-\frac{1}{b}} \leq  \sigma_i , \quad i \in \NN
%\end{equation*}
\item $\exists\gamma > 0,~i_0 \geq 1~~\text{s.t.}~~ \sigma_i \leq 2^\gamma  \sigma_{2i}$, for all $i \in \NN$. 
%\begin{equation*}
%    \exists\gamma > 0,~i_0 \geq 1~~\text{s.t.}~~ \sigma_i \leq 2^\gamma  \sigma_{2i}  \quad %\forall i\geq  i_0.
%\end{equation*}
\end{enumerate}
\end{assumption}

Since the covariance operator $\tp$ depends on the marginal $\nu$, Assumption \ref{Ass:eigen.decay} imposes a restriction on the class of marginal distributions $\cP(X)$ on $X$. We therefore introduce the following sub-classes: 
\begin{equation}
\label{eq:marginal1}
  \cP_{\text{poly}}^{<} (\beta, b):=\{ \nu \in \cP(X): \sigma_i \leq \beta i^{-\frac{1}{b}}  \;\; \forall \;\; i \in \mbn  \} \;,
\end{equation}

\begin{equation}
\label{eq:marginal2}
 \cP_{\text{poly}}^{>}(\alpha , b):=\{ \nu \in \cP(X): \sigma_i \geq \alpha i^{-\frac{1}{b}}  \;\; \forall \;\; i \in \mbn  \} \;,
\end{equation}

\begin{equation}
\label{eq:marginal3}
 \cP_{\text{str}}^{>}(\gamma):=\{ \nu \in \cP(X): \sigma_i \leq 2^\gamma  \sigma_{2i}  \;\; \forall \;\; i\geq i_0 \;,  \;\; i_0 \in \mbn\} \;.
\end{equation}

\vspace{0.3cm}

\begin{remark}[Effective Dimension]
In statistical learning theory, assumptions about marginal distributions are expressed more broadly using the \emph{effective dimension}, also known as \emph{statistical dimension}, representing the effective number of degrees of freedom. In statistical modeling, particularly with high-dimensional data, the effective dimension denotes the parameters or features significantly influencing data variation. This is crucial when the actual data dimension is high, yet numerous features may be redundant or carry little information. For $\lam >0$, it may be defined as     

\[ 
\mathcal{N}(\la):=\op{Tr}\left((\tp+\la I)^{-1}\tp\right) \;.
\]
Since $\tp$ is trace class and bounded, we immediately get 
\[ \mathcal{N}(\la) \leq \kappa^2\lam^{-1} \;.\]
However, an upper bound for the decay rate of the eigenvalues of $\tp$ implies an even sharper bound for $\cN(\lam)$. Specifically, for any $\nu \in \cP_{\text{poly}}^{<} (\beta, b)$ we find 
\[ \mathcal{N}(\la) \leq \frac{\beta b}{b-1}\la^{-b},  \]
see \cite[Proposition 3]{Caponnetto07}. Note that in our context of inverse learning, bounds on $\cN(\lam)$ also depend on $A$, incorporating properties of the inverse problem into the statistical dimension naturally. 
\end{remark}

\vspace{0.3cm}
\noindent
{\bf A-Priori Smoothness. } Another way to integrate prior knowledge into solving inverse problems is by means of \emph{(approximate) source conditions}. In statistical inverse problems, a source condition describes the regularity or smoothness of the true solution or unknown parameter being estimated. 
The choice of a source condition is motivated by the specific characteristics of the problem at hand and the available information about the true solution.

%\vspace{0.2cm}

In the context of regularization of linear ill-posed inverse problems, the a-priori-smoothness is expressed in terms of powers of the covariance operator $\tp$, to be considered as a smoothing operator: 

%In classical regularization theory, convergence rates for linear inverse problems are usually based on assumptions on the normal operator $A^* A$, both in terms of its mapping properties as well as how it can be used to characterize the set of possible solutions (source condition)~\cite{Engl96}. In statistical inverse learning, this role is given to so-called covariance operator that blends the mapping properties of $A$ with the design measure $\nu$. 

%In order to get the convergence rates, we need to make some assumptions on the unknown probability measure. The source conditions are the standard approaches to make the smoothness assumption for the true solution. In kernel methods, the source condition is defined in terms of the covariance operator but it is considered in terms of the known bounded operator $L^{-1}$ in regularization in Hilbert scales. 

%Let~$\ip$ denote the canonical injection map~$\Ht \to \LL$, where $\LL$ is the Hilbert space with the norm~$\norm{g}_\rho^2=\int_X \paren{g(x)}^2 d\nu(x)$. Then we observe that, under Assumption~\ref{Ass:Kernel}, that the operator~$\ip$ is bounded, since
%\[\norm{\ip f}_{\rho}^2=\int_X\paren{f(x)}^2d\nu(x) =\int_X\abs{K_x^*f}^2d\nu(x) \leq %\kappa'^2\norm{f}_{\Ho}^2.
%\]
%We call~$\bp := \ip A:\Ho\to\LL$ the \emph{population operator} and 
%\begin{equation}
%    \tp:= \bp^*\bp = \int_X (A^*K_x) \otimes (A^*K_x) \nu(dx):\Ho\to\Ho
%\end{equation}
%the \emph{covariance operator}. Moreover, we denote~$\lp:= L^{-1}\tp L^{-1}:\Ho\to\Ho$.
%}

\begin{assumption}[Source condition]\label{Ass:source}
Let $r>0$, $\Ro >0$. We assume that the solution~$\fp$ is a member of the class~$\Omega(r,\Ro, \tp)$ with
  \begin{equation*}
    \Omega(r,\Ro, \tp):=\left\{f \in \Ho: f= \tp^r v \text{ for some } v \in \cH \text{ with } \norm{v}_{\Ho} \leq \Ro\right\}.
  \end{equation*}
\end{assumption}
%Here, we see that the above condition is the standard source condition for the linear inverse problem~$\ip A (f)= \gp =\int_\RR y d\rho(y|x)$ for $\tp=(\ip A)^*(\ip A)$~\cite{Engl96}.

% \nic{
%\begin{example}
% The Sobolev space $H^{\alpha + d/2}(\mbr^d)$ is generated by the Matern kernel % $K_\alpha$, $\alpha >0$. Problem: What Sobolev space is generated by $T^r_\nu (H^{\alpha + d/2}(\mbr^d))$, with $\nu$ the uniform measure? 
% \end{example}
% }
% \vspace{0.2cm}

In the context of regularization schemes within Hilbert Scales, the source condition is considered in terms of the known unbounded operator~$L$, which is involved in the regularization approach. 

 \begin{assumption}[General Source Condition]
 \label{Ass:source.HS}
 For~$s>0$ and $\Ro >0$, we let
   \begin{equation*}
    \fp\in \Omega(s,\Ro,L^{-1}):=\left\{f \in \Ho: f= L^{-s}v \text{ for some } v \in \cH \text{ with }\norm{u}_{\Ho} \leq \Ro\right\}.
   \end{equation*}
\end{assumption}

Source conditions are generally difficult to verify in practice or sometimes may not be pre-specified. In this case, the concept of \emph{distance functions} can be employed to quantify the deviation from an assumed level of smoothness, serving as a benchmark. This benchmark can be chosen by the user. For a given \emph{benchmark smoothness}~$p\geq 1$, we define the corresponding distance function as specified below.

\begin{definition}[Distance Function]
\label{Defi:App.source}
  Let $R>0$, $p\geq 1$ and $L$ be a generator of a Hilbert scale. We define the \emph{distance function}~$d_p : [0, \infty)\to[0, \infty)$ by
  \begin{align}\label{eq.app.source1}
    d_p(R)=\inf\brac{\norm{L(f-\fp)}_{\Ho}:f= L^{-p}u \text{ for some } u \in \cH \text{ with }\norm{u}_{\Ho} \leq R}. 
  \end{align}
\end{definition}
The distance function is positive, decreasing, convex, and continuous for all values of $0 \leq  R < 1$. Hence, the minimizer exists. As $R$ approaches infinity, the distance function approaches  $0$, as detailed in~\cite{Hofmann06}. Note that the distance function can be bounded explicitly in terms of $\Ro>0$ provided $\fp$ satisfies a more general source condition Assumption~\ref{Ass:source.HS}. More precisely, for $s < p$ we have 
\begin{equation}\label{dist.bd}
  d_p(R) \leq (\Ro)^{\frac{p-1}{p-s}}R^{\frac{1-s}{p-s}},  
\end{equation}
see~\cite[Theorem 5.9]{Hofmann07}. If $\fp \in \cH_p$, i.e. $\fp = L^{-p}u$ for some $u \in \cH$, then obviously $d_p(R)=0$. 

\vspace{0.2cm}

In the Hilbert scale scenario when the source condition is considered in terms of $L^{-1}$, the following Link condition is essential to transfer the smoothness in terms of the covariance operator $L^{-1}\tp L^{-1}$ to the smoothness with respect to the operator~$L^{-1}$ and vice versa.

\begin{assumption}[Link condition]
\label{Ass:link}
For some $p \geq 1$, $0< a< \frac{1}{2}$ and $1 \leq c <\infty$ we have 
\begin{equation}
\label{eq:link_condition}
\norm{L^{-p}u}_{\Ho}\leq \norm{(L^{-1}\tp L^{-1})^{ap} u}_{\Ho} \leq c^{p}\norm{L^{-p}u}_{\Ho},\quad u\in\Ho.
\end{equation}
\end{assumption}

\begin{remark}
The Weyl Monotonicity Theorem~\cite[Cor. III.2.3]{Bhatia97} establishes a relation between the eigenvalues $\sigma_j(L^{-1})$ and $\sigma_j((L^{-1}\tp L^{-1})^a)$, namely,
$$\sigma_j(L^{-1})\leq \sigma_j((L^{-1}\tp L^{-1})^{a}) \leq \sigma_j(\beta L^{-1})\;, \quad j \in \mbn.$$

Hence, we observe that when the eigenvalues of $L^{-1}$ exhibit a decay rate, say, $\cO(j^{-w})$ for some $w>0$, the eigenvalues of $L^{-1}\tp L^{-1}$ will follow a decay rate of $\cO(j^{-\frac{w}{a}})$. What is more, one can show that $T_\nu$ has a spectral decay $\cO(j^{(2-\frac 1a)w})$. 

%It is pretty evident heuristically that $a$ must be smaller than $\frac{1}{2}$. Since the operator $L^{-1}\tp L^{-1}$ includes $L^{-2}$. 
\end{remark}

%%%%%%%%%%%%%%%%%%%%%%%%

\vspace{0.3cm}

{\bf Class of Models.} 
Let $\Omega \subseteq \cH$ and denote by $\cK(\Omega)$ the set of regular conditional distributions $\rho(\cdot | \cdot)$ on $\cB(\mbr) \times X$ such that \eqref{Ass:fp} and \eqref{Ass:noise} hold for some $\fp \in \Omega$. Given $r>0$, $\Ro>0$, the class of models considered in most cases below is defined by 
\begin{equation}
\label{eq:class-of-models} 
\cM_\theta:=\{ \rho(dx, dy)= \rho(dy|x)\nu(dx) \; : \; 
\rho(\cdot | \cdot) \in \cK(\Omega)\;, \nu \in \cP'\} \;, 
\end{equation}
where $\cP'$ specifies the admissible design measures. This varies depending on the regularization method, see e.g. definitions
in \eqref{eq:marginal1}-\eqref{eq:marginal3} utilized in spectral regularization. The variable $\theta$ represents parametrization for data-generating distributions and our study focuses on finding parameter choice rules that yield uniform concentration rates over the domain of $\theta$. For the rigorous definition of this idea, see next section for minimax optimality in learning. In practise, we mostly consider $\theta$ as the pair $(\Gamma, \Ro) \in \RR_+ \times \RR_+$, i.e. the noise variance and a priori norm bound to $f^\dagger$, but other parametrizations could be considered as well.

We note that while $\cM_\theta$ is specified by two main elements, the source set $\Omega$ and the admissible design measures ${\mathcal P}'$, their definitions are often intertwined and are both often related to the mapping properties of the forward mapping.
In particular, for the convex penalties we will not specify $\Omega$ and ${\mathcal P}'$ but instead consider $\cM_\theta$ generated by a set of pairs $(f^\dagger, \nu)$ satisfying suitable conditions.

%%%%%%%%%%%%%%%%%%%%%%%%%%%%%%%%%%%%%%%%%%%%%%%%%%%%%%%%%%%%%%%%%%%%%%%%%
%%%%%%%%%% Minimax Optimality 
%%%%%%%%%%%%%%%%%%%%%%%%%%%%%%%%%%%%%%%%%%%%%%%%%%%%%%%%%%%%%%%%%%%%%%%%%%

\subsection{Minimax optimality}
\label{Sec:minimax.optimal}

In statistical (inverse) learning theory, the speed at which a learning algorithm $(\zz, \lam) \mapsto \fz$ converges to the true underlying model as the sample size increases is referred to as the \emph{rate of convergence}. 
It quantifies how quickly the learned model approaches the correct solution as more data becomes available. 
The rate of convergence is a crucial concept in assessing the efficiency and performance of statistical learning methods. A faster rate of convergence implies that the estimator requires fewer observations to provide a more accurate estimate or model. 
Typically, rates are discussed by means of probabilistic tail inequalities, involving the confidence level $\eta \in (0,1]$ and the sample size $m \in \mbn$. More specifically, they take the form  
\begin{equation}\label{err.bound}
    \rho^m\left\{\norm{\fz-\fp}_{\Ho} \leq  Ca_m \log\left(\frac{1}{\eta}\right)\right\}\geq 1-\eta.
\end{equation}
Here, $\rho^m$ is the joint distribution, the function~$m\mapsto a_m$ is a positive decreasing function and describes the specific rate as~$m\to \infty$. Recall from the previous section that we restrict the class of admissible distributions to a subclass $\cM_\theta$, where $\theta$ is a parameter belonging to a parameter space $\Theta$. Then, \eqref{err.bound} allows to write  
%From the inequality~\eqref{err.bound}, we get
%$$\mathbb{P}_{\zz\in Z^m}\left\{\norm{\fz-\fp}_{\Ho} \geq C\varepsilon(m) %\log\left(\frac{1}{\eta}\right)\right\}\leq \eta$$
%which implies
\[ 
\limsup\limits_{m\to\infty}\sup\limits_{\rho\in \cM_\theta}\rho^m\left\{\norm{\fz-\fp}_{\Ho} \geq \tau a_m \right\}\leq e^{-\frac{\tau}{C}}\;,
\]
with  $\tau=C \log\left(\frac{1}{\eta}\right)$.
%, $\tau \to\infty$ for  $\eta\to 0$. $\tau >0$. 
Hence, 
\[ \lim\limits_{\tau\to\infty}\limsup\limits_{m\to\infty}\sup\limits_{\rho\in \cM_\theta}\rho^m\left\{\norm{\fz-\fp}_{\Ho} \geq \tau a_m \right\}=0.
\]
Integration immediately leads to a bound in expectation: 
\begin{align*}\label{expectaion}
\EE\sbrac{\norm{\fz-\fp}_\Ho^p}^{\frac{1}{p}}
=&\paren{\int\limits_0^\infty \rho^m\left(\norm{\fz-\fp}_\Ho^p>t\right)dt}^{\frac{1}{p}} 
\leq C_p a_m,
\end{align*}
for some $C_p < \infty$, which implies 
\begin{equation*}  \limsup\limits_{m\to\infty}\sup\limits_{\rho\in \cM_\theta}\frac{\EE\sbrac{\norm{\fz-\fp}_{\Ho}^p}^{\frac{1}{p}}}{a_m} \leq C_p <\infty.
\end{equation*}

This motivates the following definitions.  We refer to~\cite{Bauer07,Caponnetto07, Blanchard18} for more details. 

\begin{definition}[Upper Rate of Convergence] 
A family of sequences $(a_{m,\theta})_{(m,\theta)\in \NN\times\Theta}$ of positive numbers is called an \emph{upper rate of convergence} in $\LLL^p(X, \nu)$ for the reconstruction error, over the family of models 
$(\cM_\theta)_{\theta \in \Theta}$, for the sequence of estimated solutions $( f_{\zz,\la_{m,\theta}} )_{(m,\theta )\in \NN\times\Theta}$, using regularization parameters $(\la_{m,\theta})_{(m,\theta)\in \NN\times\Theta}$, if
\begin{equation*}
\sup\limits_{\theta\in\Theta}\limsup\limits_{m\to\infty}\sup\limits_{\rho\in \cM_\theta}\frac{\EE\sbrac{\norm{f_{\zz,\la_{m,\theta}}-\fp}_{\Ho}^p}^{\frac{1}{p}}}{a_{m,\theta}} <\infty.
\end{equation*} 
\end{definition}

\begin{definition}[Weak and Strong Minimax Lower Rate of Convergence] 
A family of sequences $(a_{m,\theta})_{(m,\theta)\in \NN\times\Theta}$ of positive numbers is called a \emph{weak minimax lower rate of convergence} in $\LLL^p(X, \nu)$ for the reconstruction error, over the family of models $(\cM_\theta)_{\theta \in \Theta}$, if
\begin{equation*}
\sup\limits_{\theta\in\Theta}\limsup\limits_{m\to\infty}\inf\limits_{f_{\cdot}}\sup\limits_{\rho\in \cM_\theta}\frac{\EE\sbrac{\norm{f_{\zz}-\fp}_{\Ho}^p}^{\frac{1}{p}}}{a_{m,\theta}} > 0.
\end{equation*} 
where the infimum is taken over all estimators, i.e., measurable mappings $f_{\cdot} : (X \times \RR)^m \to \Ho$. It is called a \emph{strong minimax lower rate of convergence} in $L^p$ if
\begin{equation*}
\sup\limits_{\theta\in\Theta}\liminf\limits_{m\to\infty}\inf\limits_{f_{\cdot}}\sup\limits_{\rho\in \cM_\theta}\frac{\EE\sbrac{\norm{f_{\zz}-\fp}_{\Ho}^p}^{\frac{1}{p}}}{a_{m,\theta}} > 0.
\end{equation*}   
\end{definition}

\begin{definition}[Minimax Optimal Rate of Convergence] 
\label{def:optimality}
The sequence of estimated solutions $( f_{\zz,\la_{m,\theta}} )_{m}$ using the regularization parameters $(\la_m,\theta)_{(m,\theta)\in \NN\times\Theta}$ is called \emph{weak/strong minimax optimal} in $\LLL^p(X, \nu)$ for the interpolation norm of parameter $s \in  [0,\frac{1}{2}]$, over the model family $(\cM_\theta)_{\theta \in \Theta}$, with rate of convergence given by the sequence $(a_{m,\theta})_{(m,\theta)\in \NN\times\Theta}$, if the latter is a weak/strong minimax lower rate as well as an upper rate for $( f_{\zz,\la_{m,\theta}} )_{(m,\theta )\in \NN\times\Theta}$.   
\end{definition}

%%%%%%%%%%%%%%%%%%%%%%%%%%%%%%%%%%%%%%%%%%%%%%%%%%%%%%%%%%
%%%%%%%%%%%%%%% §3 Linear Problems
%%%%%%%%%%%%%%%%%%%%%%%%%%%%%%%%%%%%%%%%%%%%%%%%%%%%%%%%%%

\section{Linear problems}
\label{sec:linear}

In this Section we present the theory for regularization of linear inverse learning problems. We begin with introducing a suitable discretization strategy based on random sampling that is able to provide practical implementations of our regularized solutions. We move forward with presenting three major classes of regularization methods: \emph{spectral regularization}, \emph{regularization by projection}, and more generally \emph{convex regularization}. For each of these methods we state convergence results and discuss their minimax optimality.

%\cite{Alberti23,Bubba21,Helin22,Rastogi23}

%%%%%%%%%%%%%%%%%%%%%%%%%%%%%%%%%%%%%%%%%%%%%%%%%%%%%%%%%%
%%%%%%%%%%%% Discretization
%%%%%%%%%%%%%%%%%%%%%%%%%%%%%%%%%%%%%%%%%%%%%%%%%%%%%%%%%%

\subsection{Discretization by Random Sampling}\label{Sec:discretization}

Discretization is a crucial step in solving inverse problems because it transforms complex, continuous mathematical models into a computationally tractable format that can be approached using numerical methods. In our framework, we employ the random sample $\paren{(x_i,y_i)}_{i=1}^m$ to construct a discretization of the operator $A$ in the following way: 
We define the 
\emph{sampling operator} $\sx: \cH' \to \mbr^m$ by 
\[ \sx g := (g(x_1), ..., g(x_m)). \]
We further define $ \bx:=\sx A:\Ho\to \RR^m$, yielding an empirical approximation to $B_\nu$ and $\tx:=\bx^*\bx$, yielding an approximation to the covariance $T_\nu$.

The empirical expected risk based on the sample is given by 
\begin{equation}
  \mathcal{E}_\zz(f) =    \frac 1m \sum_{i=1}^m \paren{[A(f)](x_i)-y_i}^2 =  \frac 1m \norm{\bx f-\yy}_2^2.
\end{equation}
Minimizing $\cE_\zz$ amounts to finding $\fhat \in \cH$ such that $\bx \fhat  = \yy$, which can be seen as a noisy, randomly discretized version of \eqref{Model} and one would expect that with increasing sample size, model \eqref{Model} will be better approximated by its empirical counterpart. However, it is not straightforward to measure the accuracy of the random approximation as the ranges of the operators defining the inverse problems are different. Considering the normal equations for both problems 
\begin{equation}
\label{eq:inv-prob}
 T_\nu \fp = B_\nu^* \gp , \quad \tx \fhat = \bx^*\yy    
\end{equation}
suggest to compare both inverse problems by bounding the norms
\[ \norm{ T_\nu - \tx} , \quad \norm{ B_\nu^* \gp - \bx^*\yy}_{\cH}  .\] 
Since both norms are random quantities, estimating the perturbation due to random noise and random sampling requires probabilistic tools, such as concentration of measure results, see~\cite{Caponnetto07,DeVito05}. We provide a result that shows that $T_\nu$ and $B_\nu^* \gp$ can be well approximated by their respective empirical 
counterparts. 

\vspace{0.2cm}

\begin{proposition}[Lemma 9 in \cite{Bauer07}]
Let $\eta \in (0,1]$ and suppose that Assumption \ref{Ass:noise} is satisfied. With probability at least $1-\eta$ we have 
\begin{align*}
\norm{ T_\nu - \tx} &\leq  \frac{2\sqrt{2}\kappa^2}{\sqrt{n}}\log(4/\eta), \\
\norm{ B_\nu^* \gp - \bx^*\yy}_{\cH} &\leq 2\kappa \left(\frac{M}{n} + \frac{\sqrt{\Sigma}}{\sqrt{n}}\right)\log(4/\eta) \;.
\end{align*}
\end{proposition}

\vspace{0.2cm}

\begin{remark}[Measure theoretic point of view]
If we let $\hat \nu = \frac{1}{m} \sum_{j=1}^m \delta_{x_j}$ denote the empirical measure of the input data, we can identify $\mathscr{L}^2(X, \hat \nu)$ with $\mbr^m$, equipped with the inner product $\scalar{\yy}{\yy'}{m} := \frac{1}{m}\sum_{i=1}^m y_i y'_i$. It is then straightforward to see that 
the operators $B_\nu$, $T_\nu$  reduce to their respective empirical counterparts when replacing $\nu$ with $\hat \nu$, i.e. $B_{\hat \nu}=\bx $ and $T_{\hat \nu} = \tx$.  
\end{remark}

%%%%%%%%%%%%%%%%%%%%%%%%%%%%%%%%%%%%%%%%%%%%%%%%%%%%%%%%%%%%%%%%%%%%%%%%%%%%%%%%%%%%%%%%%%%%%%%%%%%%%%%%%%%%%%%%%%%%%%%%%%%%%%%%%%%%%%%%%%%%%%%%%%

\subsection{Optimality of spectral regularization methods}
\label{Sec:spectral}

Regularization strategies are essential in solving inverse problems to address issues related to ill-posedness, instability, and noise sensitivity in the input data. Regularization also helps to prevent overfitting by penalizing overly complex solutions, encouraging smoother or simpler solutions that are more likely to generalize well. 

Common regularization strategies in the context of inverse problems include Tikhonov regularization (also known as ridge regression), total variation regularization, and sparse regularization. These techniques add mathematical constraints or penalties to the optimization problem, promoting solutions that are both faithful to the data and have desirable properties such as smoothness or sparsity. The choice of regularization method depends on the specific characteristics of the problem and the desired properties of the solution. In what follows, we will focus on spectral methods. 

\subsubsection{Spectral regularization methods}
\label{subSec:spectral}

%Regularization approaches have received extensive attention in the realm of ill-posed inverse problems~\cite{Engl96}. Among these, Tikhonov regularization and gradient descent stand out as the prevailing methods for formulating estimators in statistical learning scenarios. Bauer et al.~\cite{Bauer07} discussed spectral regularization in statistical learning. They demonstrated that this general definition encompasses various regularization strategies, including Tikhonov, Landweber iteration, Spectral cut-off, and Accelerated Landweber iteration. This covers the fundamental attributes that any regularization scheme must possess.

%\nic{explain why they are called spectral methods, relate the properties to approx. error and estim. error}

We give now the general definition of spectral methods, followed by classical examples.

\begin{definition}[Spectral regularization]\label{Defi:spec.regu}
We say that a family of functions~$\{\ga\}_{\lam > 0}$ with 
$ g_\lam :[0,\kappa^2]\to\RR$ is a \emph{spectral regularization scheme} or \emph{spectral filter} if there exist constants~$D<\infty$, $E<\infty$, and  $\gamma_0 < \infty$ such that for all $\lam >0$
  \begin{enumerate}[(i)]
  \item~$\tsup\abs{t \ga(t)}\leq D$,
  \item~$\tsup\abs{\ga(t)}\leq\frac{E}{\la}$ and 
  \item~$\tsup\abs{\ra(t)}\leq \gamma_0$, where $\ra(t)=1-\ga(t)t$.  
\end{enumerate}
Moreover, the maximal~$q$ satisfying the condition
\[ \tsup\abs{\ra(t)}t^q\leq\gamma_q\la^q, \] 
for some constant~$\gamma_q$ and any $\lam>0$ 
is said to be the \emph{qualification} of the regularization scheme~$\ga$.
\end{definition}

%We can understand the characteristics of the following regularization schemes by analyzing the General regularization. 

\vspace{0.2cm}

\begin{example}[Spectral cut-off]
    Consider the spectral cut-off or truncated singular value decomposition (TSVD) defined by
\begin{equation*}
\ga (t)=\left \{\begin{array}{ll}
t^{-1} &,\quad \text{if} \quad t \geq \la ;\\
0 &,\quad \text{if}  \quad t < \la
\end{array}
\right.
\end{equation*}
The qualification $q$ could be any positive number and $ D=E =\gamma_0 =\gamma_q = 1$.
\end{example}

\vspace{0.2cm}

\begin{example}[Explicit regularization: (Iterated) ridge regression]
For $l \in \NN$, the iterated Tikhonov regularization is given by the function
$$\ga (t) = \sum\limits_{i=1}^l\la^{i-1}(\la+t)^{-i}=\frac{1}{t}\paren{1-\frac{\la^l}{(\la+t)^l}}$$
In this case we have $q = l$, $D=\gamma_0 =\gamma_q = 1$, and $E = l$. For $l = 1$, this scheme corresponds to the Tikhonov regularization or ridge regression.
\end{example}

\vspace{0.2cm}

\begin{example}[Implicit regularization: Gradient decent]
Gradient descent with constant stepsize $\eta>0$, a.k.a. Landweber iteration, to the function $\ga (t) = \sum\limits_{k=1}^u \eta (1-\eta t)^{u-k}$ with $\eta  \in (0, \kappa^2]$ with the choice $\la  = (\eta u)^{-1}$. Here, $D=E =\gamma_0= 1$ and the qualification $q$ can be any positive number with $\gamma_q = 1$ if $0 < q \leq 1$ and $\gamma_q = q^q$ if $q > 1$.
\end{example}

%\begin{example}[Implicit regularization: Stochastic gradient descent (SGD)]
%Although SGD itself is not a spectral regularization method, we highlight the connection to gradient methods due to its widespread application. 
%We then may write 
%\[  \]
%In this sense, SGD can be considered as a perturbation of GD. 
%\end{example}

%%%%%%%%%%%%%%%%%%%%%%%%%%%%%%%%%%%%%%%%%%%%%%%%%%%%%%%%%%%%%%%%%%%%%%%%%%%%%%

\subsubsection{Classical Error Bounds} 
\label{subsec:linear-classical}

%The error bound can be measured in the reconstruction norm $\norm{\fz-\fp}_{\Ho}$ and the prediction norm $\norm{\ip\paren{A(\fz)-A(\fp)}}_{\nu}$.
%Blanchard et al.~\cite{Blanchard18} discussed the minimax optimal rates for the above general solution of the statistical learning problem~\eqref{Model}. In Theorem~\ref{lin.err.upper.bound}, the error estimates are bounded in terms of the sample size~$m$ and the regularization parameter~$\la$ which hold in the probabilistic sense. In Theorem~\ref{lin.rates.gen.k}, we obtain the explicit convergence rates in terms of $m$ by choosing the regularization parameter in terms of sample size by balancing the terms on the right-hand side.

Equation \eqref{eq:inv-prob} suggests to define our regularized empirical solution as given below. 

\begin{definition}
Let $\{\ga\}_{\lam > 0}$ be a family of filter functions defined on the spectrum of $T_\nu$ and let $((x_j , y_j))_{j=1}^m$ be an i.i.d. sample. The regularized empirical estimator for the solution to the linear inverse problem \eqref{Model} is defined by 
\begin{equation}
\label{eq:gen.reg}
    \fz=\ga(\tx)\bx^*\yy.
\end{equation}
\end{definition}

In what follows we present the rates of convergence for the reconstruction error for the estimator defined above.  
%For the asymptotic analysis, we require the standard assumption that relates the sample size~$m$ and the parameter~$\la$ in such a way that
%\begin{equation}
%\label{m.la.condition}
% m^{-\frac{1}{b+1}}\leq\la\leq 1. 
%\end{equation}
%The above condition says that as the regularization parameter~$\la$ decreases, the sample size must increase. A stronger assumption is considered on the relation of $m$ and $\la$ in~\cite{Blanchard18} which also involve the confidence level~$\eta$. However, we can obtain the same convergence rates in the following theorem under the assumption~\eqref{m.la.condition} using Prop. A.2~of~\cite{Blanchard19}.

\begin{theorem}[Blanchard et al.~\cite{Blanchard18}]
\label{lin.err.upper.bound}
Suppose Assumptions~\ref{Ass:fp},~\ref{Ass:noise},~\ref{Ass:kernel.linear},~\ref{Ass:eigen.decay} (i),~\ref{Ass:source} are satisfied. Suppose further that $\lam \geq m^{-\frac{1}{b+1}}$ and the qualification is $q \geq r $. 
For all~$\eta \in (0,1]$, the regularized solution~$\fz$ (\ref{eq:gen.reg}) satisfies with probability at least~$1-\eta$ the upper bound 
\begin{align}
\label{eq:ub}
\norm{\fz-\fp}_{\Ho} \; \leq & \; C\paren{ \Ro \la^r  + \frac{\Sigma}{\sqrt{m}\la^{\frac{b+1}{2}}} + \frac{\Ro}{\sqrt{m}} + \frac{M}{m\la}}\log^3\left(\frac{8}{\eta}\right),
\end{align} 
where $C$ depends on all model parameters~$r,~ \kappa,~\gamma_0,\gamma_q,~D,~E,~b,~\beta$.
\end{theorem}

Note that this upper bound is given in high probability. 
Deriving the final rate of convergence requires to carefully select the regularization parameter $\lam$, depending on the sample size and the a-priori assumptions. Setting $\theta=(\Sigma,\Ro)$, a straightforward calculation shows that the choice 
\begin{equation}
\label{lin.para.choice}
   \la_{m,\theta}= \paren{\frac{\Sigma}{\Ro\sqrt{m}}}^{\frac{2}{2r+b+1}}
\end{equation}
equalizes the first and second term of \eqref{eq:ub}. Moreover, the remaining terms can be shown to be of lower order. We summarize these observations below. 

%For the parameter choice~$\la= \paren{\frac{\Sigma}{\Ro\sqrt{m}}}^{\frac{2}{2r+b+1}}$, we obtain $\frac{1}{\sqrt{m}}=\la^r\paren{\frac{\Ro}{\Sigma}\la^{\frac{b+1}{2}}}$ and $\frac{M}{m\la}=\frac{\Sigma}{\sqrt{m}\la^{\frac{b+1}{2}}}\paren{\frac{M\Ro}{\Sigma^2}\la^{b+r}}$. For the sufficiently large~$m$, we have sufficiently small~$\la$ and consequently, we get~$\frac{1}{\sqrt{m}}\leq\la^r$ and~$\frac{M}{m\la}\leq\frac{\Sigma}{\sqrt{m}\la^{\frac{b+1}{2}}}$. Therefore, the second and the third terms are dominated by the first and the fourth terms of the error estimates in Theorem~\ref{lin.err.upper.bound}, respectively. By balancing the remaining terms~$\Ro \la^r$ and~$\frac{\Sigma}{\sqrt{m}\la^{\frac{b+1}{2}}}$, we obtain the following theorem. On the other hand, for the parameter choice~$\la= \paren{\frac{1}{\sqrt{m}}}^{\frac{2}{2r+b+1}}$ and $\la\leq 1$, we obtain $\frac{1}{\sqrt{m}}\leq\la^r$ and $\frac{1}{m\la}\leq \frac{1}{\sqrt{m}\la^{\frac{b+1}{2}}}$. Hence, by balancing the terms, we obtain the convergence rates~$\paren{\frac{1}{\sqrt{m}}}^{\frac{2r}{2r+b+1}}$ which $m$ does not need to be sufficiently large.

\begin{corollary}
\label{lin.rates.gen.k}
Suppose the assumptions of Theorem~\ref{lin.err.upper.bound} are satisfied. Choose $\lam$ according to 
\eqref{lin.para.choice}. 
With probability at least~$1-\eta$, we have 
\begin{align*}
\norm{\fz-\fp}_{\Ho}   \; \leq &   \;  C \Ro\paren{\frac{\Sigma}{\Ro\sqrt{m}}}^{\frac{2r}{2r+b+1}}\log^3\left(\frac{8}{\eta}\right),
\end{align*} 
provided that $m$ is sufficiently large and where $C$ depends on all model parameters $r$, $\kappa$, $\gamma_0$,  $\gamma_q$, $D$, $E$, $b$, $\beta$. 
\end{corollary}

Following the lines of Section \ref{Sec:minimax.optimal}, we immediately get the upper rate of convergence through integration.

\begin{theorem}\label{lin.rates.gen.u}
Suppose the assumptions of Corollary~\ref{lin.rates.gen.k} hold true. 
Then, the sequence 
\begin{align}
\label{lin.rate.seq}
a_{m,\theta} = \Ro\paren{\frac{\Sigma}{\Ro\sqrt{m}}}^{\frac{2r}{2r+b+1}}
\end{align} 
is an upper rate of convergence in $\LLL^p(X, \nu)$ for all $p > 0$, for the reconstruction error, over the family of models $\cM(r, \Ro,  \cP_{\text{poly}}^{<} (\beta, b) )$ for the sequence of solutions of the regularization scheme~\eqref{eq:gen.reg} with regularization parameter \eqref{lin.para.choice}.
\end{theorem}

%\nic{The notation of Modelclass and parameters $\theta$ is not maximally cleaned and can be further polished, but this can wait for the revision.}

Notably, the sequence given in \eqref{lin.rate.seq} can also shown to be a lower bound for the reconstruction error:

\begin{theorem}
\label{lin.rates.gen.l}
Let Assumptions~\ref{Ass:fp},~\ref{Ass:noise},~\ref{Ass:kernel.linear}, and~\ref{Ass:source} hold true. 
%where $r > 0$, $\Ro > 0$, $b > 1$ and $\al > 0$ be fixed. 
Then the sequence $(a_{m,\theta})_\theta$ defined in~\eqref{lin.rate.seq} is a weak lower rate of convergence in $\LLL^p(X, \nu)$ for all $p > 0$ for the reconstruction error, over the family of models $\cM(r, \Ro, \cP_{\text{poly}}^{>} (\alpha, b) )$. Moreover, the
sequence $(a_{m,\theta})_\theta$ is a strong  lower rate of convergence in $\LLL^p(X, \nu)$ for all $p > 0$ over the family of models $\cM(r, \Ro, \cP_{\text{str}}^{>} (\gamma) )$.
\end{theorem}

Finally, by Definition \ref{def:optimality}, we may conclude: 

\begin{theorem}
\label{theo:lin-rates-optimal}
Suppose the Assumptions of Theorem \ref{lin.rates.gen.u} and Theorem \ref{lin.rates.gen.l} are satisfied. Then the sequence of solutions $(\fz)_m$  using the regularization parameters given in \eqref{lin.para.choice} is weak and strong minimax optimal in $\LLL^p(X, \nu)$, for all $p>0$, for the reconstruction error, over the model family $\cM(r, \Ro, \cP' )$, with $\cP'= \cP_{\text{poly}}^{<} (\beta, b)  \cap \cP_{\text{poly}}^{>} (\alpha, b) )$, $\cP'= \cP_{\text{poly}}^{<} (\beta, b)  \cap \cP_{\text{str}}^{>} (\gamma) $, respectively.
\end{theorem}

%If the  lower rate coincides with the upper convergence
%rate for $\la = \la_{m,(\Sigma,\Ro)}$. Then the choice of parameter is said to be optimal. 
%For the parameter choice $\la_{m,(\Sigma,\Ro)}= \min\paren{\paren{\frac{\Sigma}{\Ro\sqrt{m}}}^{\frac{2}{2r+b+1}},1}$, Theorem~\ref{lin.rates.gen.u} share the upper convergence rate with the lower convergence rate of Theorem~\ref{lin.rates.gen.l} in $L^p$-norm.  Therefore, the choice of the parameter is optimal.
%The authors of~\cite{Blanchard18} also discussed the rates in the interpolation norm $\norm{\tp^s\paren{f-\fp}}_{\Ho}$ which implies the bound in the reconstruction norm $\norm{f-\fp}_{\Ho}$ for $s=0$ and the prediction norm $\norm{\ip\paren{A(f)-A(\fp)}}_{\nu}$ for $s=\frac{1}{2}$.

%\vspace{0.3cm}

%%%%%%%%%%%%%%%%%%%%%%%%%%%%%%%%%%%%%%%%%%%%%%%%%%%%%%%%%%%%%%%%%%%%%%%%%%%%%%

\subsubsection{Error Bounds in Hilbert Scales}

Now, we analyze the regularization schemes in Hilbert Scales for the considered statistical inverse problem~\eqref{Model}. The background of these schemes is discussed in Section~\ref{Sec:Hilbert.Scales}. To this end, let $L$ be an unbounded operator with domain $\cD(L)\subset \cH$, generating a Hilbert scale. Tikhonov regularization in Hilbert Scales is defined as:
\begin{equation}
\label{eq:Tikhonov-L}
\fz=\argmin\limits_{f\in \DD(L)}\brac{  \frac{1}{m}\sum_{i=1}^m\paren{ [Af](x_i)-y_i}^2+\la\norm{L f}_{\Ho}^2}, 
\end{equation}
with explicit solution 
$$\fz =L^{-1}(L^{-1}\tx L^{-1}+\la I)^{-1}L^{-1}\bx^*\yy.$$

Clearly, the minimizer~$\fz$ belongs to~$\DD(L)$, and formally we may introduce~$\uz:= L \fz\in\HH$. When~$\fp \in \DD(L)$, which we call the \emph{regular case}, we may set~$u^\dagger:= L\fp\in\HH$. With this notation, we can rewrite~(\ref{Model}) as
$$
\gp = A \fp = A L^{-1} u^\dagger ,\quad u^\dagger \in\DD(L).
$$
The Tikhonov minimization problem~(\ref{eq:Tikhonov-L}) would reduce to the standard one 
\begin{equation*}
%\label{mod.regu.HS}
  \min_{u \in \cH } \;\;  \frac{1}{m}\sum\limits_{i=1}^m\paren{
    [AL^{-1} u ](x_i)-y_i}^2+\la\norm{u}_{\HH}^2,
\end{equation*}
albeit for a different operator~$A L^{-1}$. The reconstruction error is given by 
$$
\norm{ \fp - \fz}_{\HH} = \norm{L^{-1}(\up - \uz)}_{\HH}.
$$
Therefore, error bounds for~$\up - \uz$ in the weak norm, i.e. in~$\HH_{-1}$, yield bounds for~$\fp - \fz$.

%Therefore, we should consider the covariance operator $\lp=(\ip A L^{-1})^*(\ip A L^{-1})$ for linear inverse problems in Hilbert scales instead of the standard covariance operator $\lp=(\ip A )^*(\ip A )$. We consider the link condition and eigenvalue decay condition for the new covariance operator. 

\vspace{0.3cm}

The above observations motivate us to define more general spectral regularization schemes from Section \ref{subSec:spectral} as given in the next definition. 

\begin{definition}
\label{def:reg-scales}%{spectral.solu}
Let $\{\ga\}_{\lam > 0}$ be a family of filter functions defined on the spectrum of $L^{-1}T_\nu L^{-1}$ and let $((x_j , y_j))_{j=1}^m$ be an i.i.d. sample. We define the estimated solutions to the linear inverse problem \eqref{Model} in the framework of Hilbert scales by 
\[  \fz =L^{-1}\ga(L^{-1}\tx L^{-1})L^{-1}\bx^*\yy \;. \]
\end{definition}

As discussed in \cite{Rastogi23}, convergence of these regularized solutions can be achieved in the original $\cH$-norm
even when no solution of (1.1) is an element of $\cD(L)$. In this case, we are in the setting of
\emph{oversmoothing}.
%In the Hilbert Scale setting, we assume that the true solution belongs to the domain~$\dl$, but it might not be true in practice. This case is very delicate to handle in the error analysis. Hence, the problem can divided into two cases: the regular case~$\fp\in \dl$ and the oversmoothing case~$\fp\notin \dl$. In the paper~\cite{Rastogi23}, the convergence rates are discussed in both the cases in the interpolation norm.  

%We present bounds for the reconstruction error for both, the regular case and the oversmoothing case in Table~\ref{comparision.1}. 

\vspace{0.2cm}

In the following theorem, the error bounds are discussed using the approximate source condition in the absence of a source condition for the regular case. As discussed above, the covariance operator in Hilbert scales is $L^{-1}\tp L^{-1}$. Therefore, we consider the eigenvalue decay condition~\ref{Ass:eigen.decay} (i) and the link condition~\ref{Ass:link} in terms of $L^{-1}\tp L^{-1}$. We assume that the eigenvalues $\sigma_i$ of the operator $L^{-1}\tp L^{-1}$ decay polynomially, i.e., $\sigma_i\leq\beta i^{-\frac{1}{b}}$. The smoothness between the operators $L^{-1}$ $L^{-1}\tp L^{-1}$ is linked by $a$ from Assumption~\ref{Ass:link}, i.e., $\range(L^{-1})=\range((L^{-1}\tp L^{-1})^a)$.

\vspace{0.2cm}

\begin{theorem}[Rastogi et al.~\cite{Rastogi23}]\label{err.upper.bound.gen.k}
Suppose Assumptions~\ref{Ass:fp}--\ref{Ass:Kernel},~\ref{Ass:eigen.decay} (i),~\ref{Ass:link} 
are satisfied and let $\eta \in (0,1]$. Let further the qualification $q \geq a(p-1) $. Define the regularized empirical solutions according to Definition \ref{def:reg-scales}. Then, with probability at least $1-\eta$ we have 
$$\norm{\fz-\fp}_{\Ho} \; \leq \;  C \la^a 
\left( d_p(R)  +   R\la^{a(p-1)} +   \frac{\Sigma}{\sqrt{m}\la^{\frac{b+1}{2}}} 
+ \frac{R}{\sqrt{m}}   +   
\frac{M}{m\la}     \right)\log^4\left(\frac{4}{\eta}\right),$$ 
where $C$ depends on the parameters $\kappa$, $\gamma_0$, $\gamma_q$, $D$, $E$, $b$, $\beta$ and $d_p(R)$ is defined in~\eqref{eq.app.source1}.
\end{theorem}

The distance function is bounded explicitly in~\eqref{dist.bd} under the source condition~\ref{Ass:source.HS}. Then, we get the explicit rate of convergence in terms of the sample size by balancing the error terms for the regularization parameter choice. These rates are presented in the last two rows of Table~\ref{comparision.1} under the different condition on the benchmark smoothness~$p$. Note that if $\fp \in \cD(L)$, we may write  $\fp=L^{-1}u$ for some $u\in\HH$. Thus, if $s< 1$, we are in the oversmoothing case while for $s \geq 1$, we are in the regular case.

\vspace{0.2cm}

Now, we summarize the results in Table~\ref{comparision.1}. We present a comprehensive picture of the convergence rates with the corresponding parameter choices and the conditions in both regular and oversmoothing scenarios. The table shows convergence rates~$a_m$ described in~\eqref{err.bound}, corresponding orders of the regularization parameter choice $\lam_m$, smoothness of the true solution $s$, and additional constraints on the benchmark smoothness~$p$. Here, $n$ is some natural number which reflects in the multiplicative constant in the error bounds. The first row represents the oversmoothing case, and the last two rows represent the regular case. In the regular case ($p>1$), the convergence rates depend on the benchmark smoothness $p$. 
%Here, the convergence rates seem to be different from Theorem~\ref{lin.rates.gen.k} obtained for the standard regularization schemes. The reason for it is that we consider assumptions with a different covariance operator in the Hilbert scales. 

In the regular case, the authors~\cite{Rastogi23} showed that the convergence rates are optimal by relating the assumptions considered in standard regularization and regularization in Hilbert scales. The optimal rates are achieved in the last row of Table~\ref{comparision.1}, provided the benchmark smoothness $p$ is set appropriately, i.e., $s\leq p\leq s+\frac{b+1}{2a}$. In contrast, the rates are sub-optimal in the oversmoothing case.

\begin{table}
  \centering {\renewcommand{\arraystretch}{2}
  \caption{Convergence rates of the regularized solution~$\fz$ for~$a< \frac{1}{2}$,~$a (p-1)\leq {q}$.}\label{comparision.1}
    \begin{tabular}{|c|l|l|l|l|}
      \hline
      \multirow{1.5}{4em}{\small{Case}}  &\small{Convergence rates}  & \small{Reg. param.} &  \small{True} & % \small{Benchmark}  &
                                                                                    \multirow{1.5}{4em}{\small{Conditions}} \\  [-10pt]
                                         & \small{ $a_m$}   &~$\la_m$ &  % \small{Smoothness} &
                                                    \small{Smooth.}  &  \\
      \hline
      \small{{\bf Oversmooth.}}  & $\Ro\paren{\frac{\Sigma}{\Ro\sqrt{m}}}^{\frac{2as}{b+1}}$ & $\paren{\frac{\Sigma}{\Ro\sqrt{m}}}^{\frac{2}{b+1}}$ & $s< 1$ &%~$p=1$ &
                                                                                                                                                 ~$a\in[ \frac{1}{n+1}, \frac{1}{2}),~n\in\NN$ 
      \\
      \hline
      \multirow{2}{4em}{\small{{\bf Regular}}} &~$\Ro\paren{\frac{\Sigma}{\Ro\sqrt{m}}}^{\frac{s}{p-1}}$ &
                                                                                    ~$\paren{\frac{\Sigma}{\Ro\sqrt{m}}}^{\frac{1}{a(p-1)}}$ & \multirow{2}{4em}{$s\geq 1$}  & % \multirow{2}{4em}{$p >1$} &
                                                                                                                                                                                                    ~$p\geq s+\frac{b+1}{2a}$ \\ 
                                         &~$\Ro\paren{\frac{\Sigma}{\Ro\sqrt{m}}}^{\frac{2as}{2as+b+1-2a}}$ &~$\paren{\frac{\Sigma}{\Ro\sqrt{m}}}^{\frac{2}{2as+b+1-2a}}$ &  % &
                                                                                                                          &~$s
                                                                                                                            \leq
                                                                                                                            p\leq
                                                                                                                            s
                                                                                                                            +\frac{b+1}{2a}$\\
      \hline
    \end{tabular}
  }
\end{table}

\vspace{0.2cm}
\begin{remark}
The regularization schemes have a saturation effect, meaning that the rates cannot improved corresponding to the smoothness of the solution beyond the qualification of the regularization scheme. Here, in the regular case, the optimal rates are established for the range $as \leq q$, provided that the scheme has qualification $q$. For standard
regularization schemes, this would correspond to the range $\frac{as}{1-2a}\leq q$, only. 
\end{remark}

%%%%%%%%%%%%%%%%%%%%%%%%%%%%%%%%%%%%%%%%%%%%%%%%%%%%%%%%%%%%%%%%%%%%%%%%%%%%%%%%%%%%%%%%%%%%%%%%%%%%%%
%%%%%%%%%%%%% Regularization by projection
%%%%%%%%%%%%%%%%%%%%%%%%%%%%%%%%%%%%%%%%%%%%%%%%%%%

\subsection{Regularization by projection}
\label{Sec:projection.method}

Regularization by projection is a classical method that has been studied widely~\cite{Engl96}. In statistical inverse learning, it was studied recently in \cite{Helin22}. Notice that in \cite{Helin22} an additive Gaussian noise model is assumed. Therefore, in this section, $\epsilon_i$ is zero-centered Gaussian with variance $\Gamma$. Here, we modify this recent result to align with the Hilbert scale structure discussed above.
To set the stage, we assume that there is a sequence of nested subspaces of $\HH$ that provide a consistent approximation scheme.
\begin{definition}\label{admissible.space}
Consider a series of finite-dimensional subspaces of $\HH$ denoted as $V_n$, $n\geq 1$ with the condition $\text{dim} V_n=n$. Additionally, let $P_n:\Ho\to V_n\subset\Ho$ represent an orthogonal projection. The sequence $\brac{V_n}_{n=1}^\infty$ is termed `admissible subspaces' if it meets the criteria that  $V_n\subset V_{n+1}$ for all $n\in\NN$ and  $\cup_{n=1}^\infty V_n=\HH$.
\end{definition}

Given training data with $m$ samples, we aim to find an approximation space $V_n$ and an estimator $f_{n,m} \in V_n$ that concentrates to the true solution $\hat f$ with a proper parameter choice rule $n = n(m)$. Let us define the maximum likelihood (ML) estimator within $V_n$ as follows:
\begin{equation}\label{emp.err}
    f_{n,m}=\argmin\limits_{f\in V_n} \norm{\sx [A (f)]-\yy}_m^2 \quad \text{a.s.}
\end{equation}
In order to ensure that the estimator $f_{n,m}$ is almost surely well-defined, i.e. the minimization problem has a unique solution, we need to further assume that the design measure is assumed to be atomless, $A$ is injective and the kernel $k$ is strictly positive-definite. For what follows, recall that $B^+$ refers to the pseudoinverse of an operator $B:\Ho \to \Ho$.

Here, we assume that the source set is specified by the approximation scheme $\{V_n\}_{n=1}^\infty$, namely, we set
\begin{equation}\label{theta}
    \Theta(s,\Ro)=\brac{f\in \HH:\norm{(I-P_n)f}_\HH\leq \Ro(n+1)^{-s} \quad \forall~~n\geq 0}\subset\HH.
\end{equation}
where $s,\Ro > 0$ and $V_n$, $n\geq 1$ are admissible subspaces defined in~\ref{admissible.space}.
With this groundwork, the probabilistic reconstruction error is defined through the subsequent theorem:

\begin{theorem}\label{thm:Proj1}
Suppose $L$ specifies a Hilbert scale according to definition \ref{def:hilbertscale}
such that the spectrum of $L$ satisfies $\sigma_j(L^{-1}) \geq D_1 j^{-\frac{t}{\gamma}}$, $t,\gamma>0$, and $V_n \subset \DD(L)$ for any $n\geq 0$.
We assume that $\nu$ satisfies Assumption \ref{Ass:link} with $p=1$ and $a<\frac 12$ such that $\gamma = 1/a-2$ and suppose $\tp$ is a Hilbert-Schmidt operator satisfying \eqref{eq:tpLgamma_cond} with ${\mathcal R}(\tp) \subset \DD(L^{\gamma})$.
Given a sequence of admissible subspaces $\brac{V_n}_{n=1}^\infty$, we assume that for some $D_2>0$ we have
\begin{equation}\label{eq:tpLgamma_cond}
    \max\left\{\sup_{n\in {\mathbb N}} \norm{(P_n \tp P_n)^+ P_n L^{-\gamma} P_n}, \norm{L^{\gamma} \tp}\right\} < D_2
\end{equation}
and
\begin{equation}
        \label{eq:projection_crossprod_assumption}
        \sup_{n\in {\mathbb N}}\norm{(P_n L^{-\gamma} P_n)^+ L^{-\gamma}} \leq D_2
    \end{equation}
If $\fp\in\Theta(s,R)$, there exists a constant $C$ depending on $D_j,~j = 1, 2$ such that for $m\geq n$ the ML estimator $f_{n,m}$ defined by identity~\eqref{emp.err} satisfies the following bound with probability greater than $1-\eta$ for $0<\eta<1$,
\begin{equation}
    \label{eq:projection_prob_concentration}
    \norm{f_{n,m}-\fp}_{\Ho}\leq C\sbrac{\Ro n^{-s}+\log\paren{\frac{8}{\eta}}\delta\paren{\frac{n^t}{m}+\sqrt{\frac{n^{t+1}}{m}}}}    
\end{equation}
provided that
$$\log\paren{\frac{8}{\eta}}\leq \frac{\sqrt{m}}{12}n^{-t}.$$
\end{theorem}

\begin{proof}
Below, we describe the result \cite[Thm. 1.2]{Helin22} and prove that conditions therein are satisfied by the our assumptions. In the following, we employ the convention $P_0 = 0$. In \cite{Helin22} the author considered following sub-classes of design measures:
$${\mathcal P}_{\text{proj}}^{>}(t,C)=\brac{\nu\in \PP(X):\sigma_{\min}(P_n\tp P_n)\geq Cn^{-t}\; \text{for all}\; n\in \NN}$$
and
$${\mathcal P}_{\text{proj}}^{\times}(C)=\brac{\nu\in \PP(X):\norm{(P_n\tp P_n)^{+}\tp(I-P_n)}_{\Ho}\leq C \; \text{for all}\;n\in \NN}$$
where $\sigma_{\min}(P_n\tp P_n)$ refers to the smallest eigenvalue of the finite-dimensional operator $P_n\tp P_n : V_n \to V_n$. Now \cite[Thm. 1.2]{Helin22} states that if $f^\dagger \in \Theta(s,R)$ and $\nu \in {\mathcal P}_{\text{proj}}^{>}(t,C_1) \cap {\mathcal P}_{\text{proj}}^{\times}(C_2)$, then inequality \eqref{eq:projection_prob_concentration} holds with probability
$$\log\paren{\frac{8}{\eta}}\leq \frac{\sqrt{m}}{12}\sigma_{\min}(P_n\tp P_n)$$ for a constant $C>0$ depending on $C_1$ and $C_2$.

Let us now deduce that our assumptions imply $\nu \in {\mathcal P}_{\text{proj}}^{>}(t,C) \cap {\mathcal P}_{\text{proj}}^{\times}(C)$ for some $C>0$ depending on $D_1$ and $D_2$.
Applying the operator concave function~$t\mapsto t^{1/2a}$ to inequality \eqref{eq:link_condition} in assumption \ref{Ass:link} respects the partial ordering, and we obtain that
\begin{equation}
\label{eq:Lnu1}
\scalar{L^{-1} \tp L^{-1} u}{u}{\HH} \geq \scalar{L^{-1/a}u}{u}{\HH}.
\end{equation}
Suppose~$u:= L P_n v\in\HH$ and notice that by assumption $P_nv\in\DD(L)$. Then we observe that
\begin{equation*}
\scalar{\tp P_n v}{P_n v}{\HH} \geq \scalar{L^{-\gamma} P_n v}{P_n v}{\HH}.
\end{equation*}
Since $\gamma>0$, the operator~$L^{-\gamma}$ is bounded, and by the Weyl Monotonicity Theorem~\cite[Cor.~III.2.3]{Bhatia97} we see that
\begin{equation}\label{eq.sing1}
\sigma_n\paren{P_n\tp P_n}\geq \sigma_n\paren{P_n L^{-\gamma} P_n}.
\end{equation}
In particular, by our assumption on the spectral decay of $L^{-1}$ we have that
\begin{equation}
    \sigma_{\min}\paren{P_n\tp P_n} \geq \sigma_{\min} \paren{P_n L^{-\gamma} P_n} \geq D_1 n^{-t}.
\end{equation}

Next, we observe that
\begin{eqnarray}
    \norm{(P_n \tp P_n)^+ \tp (I-P_n)}_{\Ho}
    & = & \norm{(P_n \tp P_n)^+ P_n L^{-\gamma} P_n (P_n L^{-\gamma} P_n)^+ L^{-\gamma} L^{\gamma} \tp (I-P_n) }_{\Ho} \nonumber \\
    & \leq & \norm{(P_n \tp P_n)^+ P_n L^{-\gamma}P_n}_{\Ho} \norm{(P_n L^{-\gamma} P_n)^+ L^{-\gamma}}_{\Ho} \norm{L^{\gamma} \tp}_{\Ho}.\label{eq:proj_aux_upper1}
\end{eqnarray}
Each of the three terms on the right-hand side of \eqref{eq:proj_aux_upper1} is bounded due to \eqref{eq:tpLgamma_cond} and \eqref{eq:projection_crossprod_assumption}. In consequence, we obtain the result applying \cite[Thm. 1.2]{Helin22}.
\end{proof}

We introduce a non-linear truncated estimator $T_{\widehat R}:\Ho\to\Ho$ such that
\[
    T_{\widehat R}(f)= 
\begin{cases}
    f,& \text{if } \norm{f}_{\Ho}\leq {\widehat R}\\
    0,              & \text{otherwise}
\end{cases}
\]
where a fixed non-negative value ${\widehat R} \geq 0$ is involved. We formulate a non-linear estimator $g_{n,m}^{\widehat R}$ as follows:
\begin{equation}\label{trunc.est}
    g_{n,m}^{\widehat R}=T_{\widehat R}(f_{n,m})
\end{equation}
where the selection of $\widehat R$ depends upon the parameters outlined below. Now, we establish the upper convergence rate in $L_p$  for the estimator $g_{n,m}^{\widehat R}$.

\begin{theorem}\label{Thm:ML2}
Suppose $L$ specifies a Hilbert scale according to definition \ref{def:hilbertscale}
such that the spectrum of $L$ satisfies $\sigma_j(L^{-1}) \geq D_1 j^{-\frac{t}{\gamma}}$, $t,\gamma>0$, and $V_n \subset \DD(L)$ for any $n\geq 0$.
Let ${\mathcal P}'$ be the set of probability measures $\nu$ that 
\begin{enumerate}
    \item satisfy Assumption \ref{Ass:link} with $p=1$ and $a<\frac 12$ with $\gamma = 1/a-2$, 
    \item for which $\tp$ is a Hilbert-Schmidt operator satisfying \eqref{eq:tpLgamma_cond} and \eqref{eq:projection_crossprod_assumption} with some $D_2>0$, and
    \item ${\mathcal R}(\tp) \subset \DD(L^{-\gamma})$.
\end{enumerate}
 Let $2s - t + 1 > 0$ and set $g_{n,m}^{\widehat R}$ as the ML estimator defined by identity~\eqref{trunc.est}. For the sequence of solutions corresponding to the parameter choice 
\begin{equation}\label{sam.size}
    n=\paren{\frac{\Sigma}{\Ro\sqrt{m}}}^{-\frac{2}{2s+t+1}}
\end{equation}
and $\widehat R= \widehat R(n,\Sigma)$  given by 
\begin{equation}\label{R.para}
   \widehat R=C\paren{\frac{\Sigma}{\sigma_{\min}(P_n\tp P_n)}+1} 
\end{equation}
with suitably large constant $C$ depending on $D_2$ and $R$, 
the sequence 
\begin{equation*}
    a_{m, R,\Sigma}=\Ro\paren{\frac{\Sigma}{\Ro\sqrt{m}}}^{\frac{2s}{2s+t+1}}
\end{equation*}
is an upper rate of convergence in $L^p$ for all $p>0$ for the reconstruction error over the family of $\mathcal{M}(\Theta(s,\Ro), \cP')$.
\end{theorem}

Let us also note that the paper \cite{Helin22} includes a proof of minimax optimality regarding the above rates obtained using the class of models specified by ${\mathcal P}_{\text{proj}}^{>}$, ${\mathcal P}_{\text{proj}}^\times$ and a set ${\mathcal P}_{\text{proj}}^{<}$ corresponding to an upper bound to the spectral decay. However, extending the proof of lower bounds to the Hilbert scale setting is beyond the scope of this paper.

%%%%%%%%%%%%%%%%%%%%%%%%%%%%%%%%%%%%%%%%%%%%%%%%%%%%%%%%%%%%%%%%%%%%%%%%%%%%%%%%%%%%%%%%%%%%%%%%%%%%%%
%%%%%%%%%%%%% Convex regularization
%%%%%%%%%%%%%%%%%%%%%%%%%%%%%%%%%%%%%%%%%%%%%%%%%%%

\subsection{Convex regularization}\label{Sec:Convex.reg}

Convex variational regularization techniques are of interest in various inverse problems, where certain features of the unknown, such as edges in imaging or sparsity in signal processing, needs to be emphasized.
In this section, we deviate briefly from the earlier setup and assume that the forward linear operator $A$ maps a separable Banach space ${\mathcal B}$ to a RKHS $\cH'$.

Below, we discuss the characteristics of regulated solutions  $\fz$ obtained through the variational problem
\begin{equation}
    \label{eq:convex_optimization}
    \fz=\argmin\limits_{f\in\HH} \brac{\frac{1}{2m}\sum\limits_{i=1}^m\paren{A(f)(x_i)-y_i}^2+\lambda G(f)}
\end{equation}

Here, $G:{\mathcal B} \to \RR\cup\brac{\infty}$ represents a convex functional. Let us point out that penalties such as 
\begin{equation}
    \label{eq:convex_penalty}
    G(f) = \frac 1p \norm{f}_{{\mathcal B}}^p    
\end{equation}
have been studied in statistical learning (see e.g. \cite{mendelson2010, steinwart2009optimal}).
However, in the inverse problems context much less is known and to the best of our knowledge their study with random design is so far limited to \cite{Bubba21, benning2023trust}.

The first article \cite{Bubba21} leverages ideas from deterministic convex regularization theory, where Bregman distances \cite{benning2018modern} have become popular during last decade. While those techniques extend to more general penalties, here, we restrict to the $p$-homogeneous functional \eqref{eq:convex_penalty}. Due to strict convexity for $p\geq 1$, the regularized solution $\fz$ is unique. In addition, the mappings
\begin{equation}
    (f, {\bf z}) \mapsto \frac{1}{2m}\sum\limits_{i=1}^m\paren{A(f)(x_i)-y_i}^2+\lambda G(f) \quad \text{and}\quad \bz \mapsto \fz
\end{equation}
are continuous and measurable, respectively (see comment in \cite[Sec. 4]{Bubba21}). Therefore, the learning problem in \eqref{eq:convex_optimization} becomes well-posed. Notice that \cite{Bubba21} assumes more general range (Banach space) of operator $A$ with similar continuous pointwise evaluation functionals. As we limit to RKHS-valued range, this simplifies the assumptions needed in \cite{Bubba21}.

The optimality criterion associated with \eqref{eq:convex_optimization} is given by
\begin{equation}
    \bx^*(\bx \fz - \yy) + \lambda r_{{\bf z},\lambda} = 0
\end{equation}
for $r_{{\bf z},\lambda} \in \partial G(\fz)$, where the dual operator $\bx^*$ maps from the RKHS $\cH'$ to the Banach dual ${\mathcal B}^*$ and $\partial G$ denotes the subdifferential set
\begin{equation*}
    \partial G(f) = \{r \in {\mathcal B}^* \; | \; G(f) - G(\tilde f) \leq \langle r, f- \tilde f\rangle_{{\mathcal B}^*\times {\mathcal B}} \; \text{for all} \; \tilde f\in {\mathcal B}\}.
\end{equation*}
For $r_f \in \partial G(f)$ and $r_{\tilde f} \in \partial G(f)$ we define symmetric Bregman distance between $f$ and $\tilde f$ as
\begin{equation*}
    D_G^{r_f, r_{\tilde f}}(f,\tilde f) = \langle r_f - r_{\tilde f}, f - \tilde f\rangle_{{\mathcal B}^* \times {\mathcal B}}
\end{equation*}
Following \cite{Bubba21} we restrict to the case $1<p<2$. In consequence, the subdifferential elements $r_f\in \partial G(f)$ are unique and in what follows we drop the dependence of Bregman distance on the subgradients and simplify notation writing $D_G(f,\tilde f)$.

The crux of analysing the non-quadratic objective functional in \eqref{eq:convex_optimization} is that it does not give rise to similar explicit formula for the minimizer $\fz$ as e.g. in the Tikhonov case in \eqref{eq:gen.reg}. However, as is shown in \cite{Bubba21}, the Bregman-based reconstruction error can be separated into bias-variance type of decomposition. The bias is represented by the term
\begin{equation}
    {\mathcal R}(\beta, \bz; f^\dagger) =
    \inf_{{\bf w} \in \RR^m} \left[G^\ast(\Ro - B_{\bz}^* {\bf w}) + \frac \beta {2m} \norm{{\bf w}}^2\right],
\end{equation}
where $G^\ast$ is the convex conjugate of $G$.
The functional ${\mathcal R}$ is closely related to the distance function in  Definition \ref{Defi:App.source}. However, notice carefully that unlike the distance function, the functional ${\mathcal R}$ is dependent on the random sampling. Therefore, conditions for the source set $\Theta$ and set of design measures ${\mathcal P}'$ become intertwined without further assumptions on the structure of $G$.
Moreover, the error variance is represented by $G^\ast(B_{\bz}^* {\bf \epsilon})$. The general concentration result is given as follows:

\begin{theorem}[{\cite[Thm. 4.11]{Bubba21}}]
\label{thm:convex}
Let $G$ be defined in \eqref{eq:convex_penalty} and let ${\mathcal M}_\theta$, $\theta = (\Sigma, \Ro)$, consist of all data-generating distributions induced by conditions \begin{equation*}
    G(f^\dagger) \leq \Ro
\end{equation*}
and
\begin{equation*}
    \EE {\mathcal R}(\beta, \bz; f^\dagger) \leq D_1 \beta + D_2 m^{-Q}
\end{equation*}
for some $\Ro, Q, D_1, D_2>0$.
Moreover, let us assume that
\begin{equation*}
    \EE G^\ast(B_{\bz}^* {\bf \epsilon}) \leq D_3 m^{-\frac q2}
\end{equation*}
for some $D_3>0$, where $q$ is the H\"older conjugate of $1<p<2$.
Then it follows that for the parameter choice rule
\begin{equation}
\label{eq:param-convex1}
    \lambda = \left(\frac{D_3^{\frac 2q} \Ro^{-\frac{q+2}{q}}}{D_1}\right)^{\frac 13} \left(\frac{\Sigma}{\Ro \sqrt m}\right)^{\frac 23}
\end{equation}
we have
\begin{equation}
     \sup_{(\Sigma, \Ro) \in \RR^2}\limsup_{m\to\infty} \sup_{\rho \in {\mathcal M}} \frac{\EE D_G(\fz,f^\dagger)}{a_{m,\Ro, \Sigma, p}} \leq C,
\end{equation}
where the concentration rate is given by
\begin{equation}
a_{m,\Ro, \Sigma, p} = \left(\frac{\Sigma }{\Ro \sqrt m}\right)^{\frac 23}
\end{equation}
and the constant $C$ depends on $D_j$, $j=1,2,3$, and $p$.
\end{theorem}

For a more concrete result, let us consider wavelet decompositions in $\RR^d$. 
Let ${\mathcal B} = B_{pp}^s(\RR^d) := B_p^s(\RR^d)$ be a Besov space \cite{daubechies1992ten} and
\begin{equation}
\label{eq:besov_penalty}
    G(f) = \frac 1p \norm{f}_{B_p^s}^p := \frac 1p \sum_{\alpha=1}^\infty c_{\alpha, p, s,d} |\langle f, \psi_\alpha\rangle|^p
\end{equation}
for some $1<p<2$, where
\begin{equation}
    \label{eq:wavelet_coefs}
    c_{\alpha, p, s,d} = 2^{|\alpha|d\left(p\left(\frac sd +\frac 12\right)-1\right)}.
\end{equation}
Here, $\psi_\alpha : \RR^d \to \RR$, with $\alpha=1,...,\infty$, are suitably regular functions that form an orthonormal wavelet basis for $\mathscr{L}^2(\RR^d)$ with global indexing $\alpha$. The notation $|\alpha|$ is used to denote the scale of the wavelet basis associated with the index $\alpha$. Rephrasing \cite[Cor. 5.6]{Bubba21} in terms of the upper rate of convergence yields the following statement.

\begin{theorem}
\label{thm:besov}
Let $G$ be defined by \eqref{eq:convex_penalty} and let the model set ${\mathcal M}$ be induced by conditions
$G(f) \leq \Ro$ and
\begin{equation*}
    \partial G(f) = \{B^*_\nu w\} \; \text{for}\; \norm{w} \leq \Ro.
\end{equation*}
Moreover, assume
\begin{equation*}
    \sum_{\alpha=1}^{\infty} c_{\alpha, q, -s, d} \norm{A\psi_\alpha}_{\cH'}^q < \infty,
\end{equation*}
where coefficients $c_{\alpha, q, -s, d}$ are defined by \eqref{eq:wavelet_coefs}.
Then for parameter choice rule
\begin{equation}
\label{eq:parameter-convex}
    \lambda = \Ro^{\frac{2}{3p} - \frac 13} \left(\frac{\Sigma}{\Ro \sqrt m}\right)^{\frac 23}
\end{equation}
we have
\begin{equation}
    \sup_{(\Sigma, \Ro) \in \RR^2}\limsup_{m\to\infty} \sup_{\rho \in {\mathcal M}} \frac{\EE D_G(\fz,f^\dagger)}{a_{m,\Ro, \Sigma, p}} \leq C 
\end{equation}
with the concentration rate
\begin{equation}
a_{m,\Ro, \Sigma, p} = \Ro^{\frac{2+5p}{3p}}\left(\frac{\Sigma}{\Ro \sqrt m}\right)^{\frac 23}.
\end{equation}
\end{theorem}

\begin{remark}
As mentioned above, convex penalties in inverse problems, such as Lasso regularization, are often motivated by sparsity-type assumption regarding the unknown. While
Theorems \ref{thm:convex} and \ref{thm:besov} provide first attempt to formalize theory of convex penalties in statistical inverse learning with random design, the highly interesting case of $p=1$ remains open and is not easily answered by the Bregman distance based techniques. Moreover, the Bregman distance for $1$-homogeneous functional (consider e.g. the $\ell^1$-norm) is typically weak as the subgradient set is not a singleton. We point out that variational source conditions have been successful in proving concentration rates in the norm for deterministic design in statistical inversion \cite{weidling2020optimal} and it would be tempting to use similar ideas for the random design problem.
\end{remark}

%%%%%%%%%%%%%%%%%%%%%%%%%%%%%%%%%%%%%%%%%%%%%%%%%%%%%%%%%%%%%%%%%%%%%%%%%%%%%%%%%%%%%%%%%%%%%%%%%%%%%%%%%%%%%%%%%%%%%%%%%%%%%%%%%%%%%%%%%%%%%%%%%%%%%%%%%

\subsection{Parameter choice rules and adaptivity}

We recall that while tuning the regularization parameter is essential for regularization to work well, an a-priori choice of the regularization parameter is in general not feasible
in statistical problems since the choice necessarily depends on unknown structural properties such as smoothness of $f^\dagger$ or eigenvalue decay, see e.g. equations \eqref{lin.para.choice}, \eqref{sam.size},  Table \ref{comparision.1} or equations \eqref{eq:param-convex1}, \eqref{eq:parameter-convex}. 
Parameter choice rules, also known as \emph{model selection} or \emph{hyperparameter selection}, 
are crucial in the context of inverse problems
and play a vital role in finding stable, 
reliable, and interpretable solutions, that depend on the data only and are free from additional unknown a-priori assumptions.

The aim is to construct estimators, i.e. to find a sequence of regularization parameters 
$(\lam_\bz)_m$, without knowledge of a-priori assumptions, but depending
on the data $\bz$, such that the sequence of regularized solutions $(f_{\bz,\lam_\bz })_m$ is minimax optimal. Such a sequence of estimators is called \emph{minimax optimal adaptive}.

\vspace{0.2cm}

We give a very brief overview about some classical model selection criteria used in inverse problems.

\vspace{0.3cm}

{\bf Unbiased risk estimation (URE).}
A very natural idea is to  estimate the unknown risk by a functional $U(\bz , \lam)$ of the observations, and then to minimize this estimator of the risk. Ideally, $U$ is an \emph{unbiased} estimator of the risk, meaning that 
\[ \cE(f_{\bz, \lam}) = \mbe[U(\bz , \lam)] \;, \quad \forall \; \; \lam > 0 \;. \]
The principle of URE suggests to minimize over $\lam \in \Lambda$, where $\Lambda$ is a well-chosen finite grid of possible values for the regularization parameter: 
\[ \lam_\bz := \argmin_{\lam \in \Lambda} U(\bz , \lam)  \;. \]
For inverse problems, this method was studied in \cite{cavalier2002oracle}, where exact oracle inequalities for the mean square risk were obtained.
% \cite{cavalier2008nonparametric} (and references therein)

\vspace{0.3cm}

{\bf Hold-out and (generalized) cross-validation (GCV).} 
Cross-validation is a classical approach to select hyperparameters that is based on \emph{data-splitting}. The idea behind GCV is to use the data to both train and validate the model, but with a focus on finding the optimal hyperparameters. The process involves iteratively fitting the model with different hyperparameter values and evaluating the model's performance using a specific criterion. The simplist version is the \emph{hold-out method}, where the available data are split into a training set $\bz^t$ of size $m_1$ and validation set $\bz^v$ of size $m_2$, with $m=m_1 + m_2$. The estimator $f_{\bz^t, \lam}$ is defined by means of the training data and is evaluated using the empirical risk with respect to the validation data: 
\[ \lam_{\bz} :=  \argmin_{\lam \in \Lambda} \frac{1}{m_2}\sum_{j=1}^{m_2} ( f_{\bz^t, \lam}(x_j^v) - y_j^v  )^2     \;. \]
This method works well in practice and \cite{caponnetto2010cross} show minimax optimal rates for the $\LL$ error under polynomial eigendecay and a H\"older source condition. 

\vspace{0.3cm}

{\bf Lepski's balancing principle.} 
The \emph{Lepski principle}, a.k.a. the \emph{Goldenshluger-Lepski method} was introduced in the seminal paper \cite{lepskii1991problem} and  does not use any data splitting. It is merely based on comparing estimators for different levels of regularization. 
Surprisingly, the monotonicity of the approximation error suffices to provide the knowledge to choose the regularization parameter adaptively, provided an empirical approximation of the sample error      $ \cS_\bz$ is available. More precisely, let $\Lambda$ be a finite grid of possible values for $\lambda$. Set 
\[  
J_{\bz}(\Lambda) = 
\left\{ \lambda \in \Lambda : \left\| (\bx + \lambda')^{\frac{1}{2}}( f_{\bz, \lambda } - f_{\bz, \lambda'})\right\|_{\cH} \leq c \;  \sqrt{\lambda' } \; \cS_\bz( \lambda') , \quad \forall \lambda' \in \Lambda, \, \lambda' \leq \lambda 
\right\} \;.
\] 
The parameter choice based on Lepski's principle is given by 
\[ {\lambda}_{\bz} :=  \max_{\lam \in \Lambda}J_{\bz}(\Lambda)  \;. \]
In the context of learning theory, this method was analyzed in 
\cite{de2010adaptive}, \cite{lu2020balancing} and was finally improved by  \cite{blanchard2019lepskii}. 
This parameter choice rule is minimax optimal adaptive under various types of source conditions and eigenvalue decay.

\vspace{0.3cm}

{\bf Discrepancy principle and early stopping rules.} 
The above mentioned criteria require the computation of all estimators to be compared against each other. 
For problems with high dimensions, this can incur a computationally prohibitive cost. 
An alternative approach is the use of early stopping rules based on the discrepancy principle. These rules terminate the procedure at an iteration $t_\bz$  based 
solely on iterates with index $t \leq  t_\bz$ and potentially other quantities computed up to that point. Because they require 
the calculation of significantly fewer iterates, early stopping rules offer the potential to achieve both computational 
and statistical efficiency simultaneously. The stopping criterion is defined based on a validation performance: If the performance does not improve for a certain number of consecutive evaluations, 
the training process is halted. A popular validation criterion, for instance, is based on the smoothed empirical residuals 
\[ \hat R_{t, s}:= || (\bx \bx^*)^{s}(\by - \bx f_{\bz,t}) ||_2 \;, \]  
where $s \geq 0$ and $f_{\bz,t}$ is an iteratively defined estimator, as such as the last iterate of a gradient based method or truncated 
SVD, see e.g. \cite{yao2007early, stankewitz2020smoothed, blanchard2018early, blanchard2018optimal}. The iterative 
process is stopped at the smallest $t_\bz$, such that 
\[  \hat R_{t_\bz, s} \leq \tau \;,\]
for some pre-specified $\tau > 0$. Provided that the inverse problem is moderately ill-posed, 
it is known that this criterion leads to optimal adaptive bounds for $f_{\bz, t_{\bz}}$ for moderate smoothing, i.e. $s <b/2$, over Sobolev-type ellipsoids \cite{stankewitz2020smoothed}.

\vspace{0.3cm}

We finally mention some recent results related to establishing data-driven approaches for finding the optimal regularization parameter, highlighting particularly the current need of a better understanding of such methods in modern machine learning: \cite{alberti2021learning}  focus on the
case where the regularization functional is not given a priori, but learned from data, \cite{chirinos2023learning} study a statistical learning approach,
based on empirical risk minimization, and \cite{de2022machine} propose a novel method based on supervised learning to approximate the high-dimensional function, mapping noisy data into a good approximation to the optimal Tikhonov regularization parameter.

\vspace{0.3cm}

%%%%%%%%%%%%%%%%%%%%%%%%%%%%%%%%%%%%%%%%%%%%%%%%%%%%%%%%
%%%%%%%%%%%% NON-LINEAR PROBLEMS
%%%%%%%%%%%%%%%%%%%%%%%%%%%%%%%%%%%%%%%%%%%%%%%%%%%%%%%%

\section{Non-linear problems}
\label{sec:non-linear}

In this section, we discuss the statistical learning problem~\eqref{Model} when $A$ is a nonlinear operator. However, in contrast to Section \ref{sec:linear}, we only focus on Tikhonov regularization. We present classical error bounds for the reconstruction error in Section \ref{sec:non-lin-1} as well as bounds for Tikhonov regularization in Hilbert scales in Section \ref{sec:non-lin-2}. 

\vspace{0.2cm}

For analyzing the non-linear inverse problem \eqref{Model} we require  the  forward operator to be  Fr\'{e}chet differentiable (an assumption that will even be further relaxed below).

\begin{assumption}[Non-linearity of the operator I]
\label{Ass:A.oper}
We impose the following assumptions: 
\begin{enumerate}[(i)]
\item The domain $\mathcal{D}(A)$ is convex with nonempty interior.
\item The operator $A$ is Fr\'{e}chet differentiable.
\item The Fr\'{e}chet derivative~$A'(f)$ of~$A$ at~$f$ is bounded in a ball~$\mathcal{B}_d(\fp)$ of radius~$d:=4\norm{\fp-\fbar}_{\Ho}$, i.e., there exists~$\mu_1 < \infty$ such that
~$$
    \norm{A'(f)}_{\Ho\to\Ht}\leq \mu_1 \qquad \forall~f\in
    \mathcal{B}_d(\fp)\cap\mathcal{D}(A)\subset\Ho.
~$$
\end{enumerate}
\end{assumption}

Recall from Example \ref{ex:non-linear}, the definition of the covariance operator 
\[ \tp:=\bp^*\bp, \quad  \bp=\ip A'(\fp) .  \]

%In the nonlinear scenario, by linearizing the operator equation~\eqref{lin.equ} at~$\fp$, we get 
%\begin{equation}\label{nonlin.equ}
%   \ip A'(\fp)(f) = \gp.
%\end{equation}

%From the analogy of the equations~\eqref{lin.equ} and~\eqref{nonlin.equ}, we deduce that the operators defined in Section~\ref{Sec:math.prelim} should be modified by replacing $A$ to the Fr{\'e}chet derivative $A'(\fp)$ in the nonlinear case. In this case, the population and empirical operators are defined in terms of the Fr{\'e}chet derivative $A'(\fp)$. Hence, the population version operators are~$\bp=\ip A'(\fp)$,~$\tp=\bp^*\bp$,~$\lp=L^{-1}\tp L^{-1}$. In Theorem~\ref{convergence}, the source condition, and eigenvalue decay condition are defined in terms of the covariance operator~$\tp=(\ip A'(\fp))^*(\ip A'(\fp))$. In Theorems~\ref{err.upper.bound},~\ref{err.upper.bound.p}, the source condition, eigenvalue decay condition, and the link condition are defined in terms of the covariance operator~$\lp=(\ip A'(\fp)L^{-1})^*(\ip A'(\fp)L^{-1})$.

%%%%%%%%%%%%%%%%%%%%%%%%%%%%%%%%%%%%%%%%%%%%%%%%%%%%%%%%%%%%%%%%%

\subsection{Classical Error Bounds}
\label{sec:non-lin-1}

The Tikhonov regularization scheme for this problem is again given by an optimization problem 
\begin{equation}
\label{eq:Tikhonov-standard}
  f_{\zz,\la} =  \argmin_{f\in\mathcal{D}(A)\subset\Ho}\left\{\frac{1}{m}\sum\limits_{i=1}^m\paren{[A(f)](x_i)-y_i}^2+\la \norm{f-\fbar}_{\Ho}^2\right\}. 
\end{equation}
Here, $\fbar\in\cH$ denotes some initial guess of the solution $\fp$, which offers the possibility to incorporate {\em a-priori} information. Note, however, due to the non-linearity of $A$, the solution to the minimization problem is not explicitly given and does not need to be unique.

%We first discuss classical error bounds for the reconstruction error of the regularization scheme 
%\eqref{eq:Tikhonov-standard}. 

To establish our error bounds we need additional assumptions on the non-linearity of $A$. 
%\label{it:gamma}

\begin{assumption}[Non-linearity of the operator II]
\label{Ass:A.oper-2}
We assume:  
\begin{enumerate}[(i)]
\item The operator  $A:\mathcal{D}(A)\subset\Ho\to \Ht\hookrightarrow\LL$ is weakly sequentially closed and one-to-one. 
\item There exists~$\mu_2\geq 0$ such that for
all~$f\in \mathcal{B}_d(\fp)\cap\mathcal{D}(A)\subset\Ho$ we have
$$\norm{A(f)-A(\fp)-A'(\fp)(f-\fp)}_{\nu} \leq \frac{\mu_2}{2}\norm{f-\fp}_{\Ho}^2.$$  
\end{enumerate}
\end{assumption}

The first assumption ensures that there exists a global minimizer of the Tikhonov functional~\eqref{eq:Tikhonov-standard}. If $\da$ is weakly closed (closed and convex) and $A$ is weakly continuous, then weak sequential closedness is guaranteed.  On the other hand, the second assumption helps to control the remainder term of the linearized problem. The second condition also holds true when the Fr{\'e}chet derivative is Lipschitz in the operator norm (see \cite[Chapt.~10]{Engl96}).

As in the linear case, we impose a source condition for $\fp$ to facilitate convergence.  

\begin{assumption}[Source condition]
\label{ass:source-non-lin}
Let $\bar f \in \cH$. We assume that $\fp - \bar f \in \Omega(r, \Ro , T_\nu)$,  for $r \in [\frac{1}{2},1]$ and where the source set is defined in Assumption \ref{Ass:source}.   
\end{assumption}

Under the source condition~$\fp-\fbar=\tp^r u$ for $r \in [\frac{1}{2},1]$, we have that~$\fp-\fbar=\tp^{\frac{1}{2}}w$ for $w=\tp^{r-\frac{1}{2}}u$ satisfying  
\[ \norm{w}_{\HH_1}\leq \norm{T_\nu}^{r - \frac{1}{2}} \cdot \norm{u}_\cH \leq \kappa^{2r-1} \Ro 
\leq \bar\kappa \Ro =:\mu_3 \;,  \]
with $\bar \kappa = \max\{1, \kappa\}$. We additionally assume that 
\begin{equation}\label{smallness}
\Ro < \frac{1}{2\mu_2 \bar\kappa},
\end{equation}
implying that $ 2\mu_2  \mu_3<1$. This assumption is a smallness condition that imposes a constraint between $\mu_3$ and the non-linearity as measured by the parameter~$\mu_2$ in Assumption~\ref{Ass:A.oper} (iii). It ensures that the initial guess is close enough to the true solution and the residual error on linearizing the nonlinear operator~$A$ at the solution is sufficiently small. 

We need an additional restriction to achieve the lower bounds:
\begin{assumption}\label{Ass:lower}
Let us define~$T=A'(f)^*\ip^*\ip A'(f)$ and~$\overline{T}=A'(\fbar)^*\ip^*\ip A'(\fbar)$. There exists a family of bounded linear operators $R_{f}: \Ho \to \Ho$ and $\zeta>0$ such that
\begin{equation*}\label{add.assumption}
T^r=R_{f}\overline{T}^r \text{  and  } \norm{R_{f}-I} \leq \zeta\norm{f-\fbar}_{\Ho},
\end{equation*} 
where~$f$ belongs to the sufficiently large ball and $\fbar$ is the initial guess.
\end{assumption}

\begin{theorem}[Rastogi et al.~\cite{Rastogi20}]
\label{convergence}
Suppose Assumptions~\ref{Ass:fp}--\ref{Ass:Kernel},~\ref{Ass:eigen.decay}, \ref{Ass:A.oper}--\ref{ass:source-non-lin} are satisfied. Assume further that 
\begin{equation}
\label{l.la.condition.k}
8\kappa^2\max \left\{ 1,\frac{\mu_1(M+\Sigma)}{\kappa d}\right\} \log\paren{\frac{4}{\eta}} \leq \sqrt{m}\la.
\end{equation}
Let \eqref{smallness} hold true with $r\in[\frac{1}{2},1]$. Then, for all~$\eta \in (0,1]$, for a regularized least-squares estimator~$\fz$ in~(\ref{eq:Tikhonov-standard}), the following upper bound holds with probability at least $1-\eta$:
\[ \|\fz-\fp\|_{\Ho} \; \leq \;  C \left( \Ro\la^r 
+  \frac{\Sigma}{\sqrt{m}\la^{\frac{b+1}{2}}}  +  \frac{M}{m\la}   \right)  \log\left(\frac{4}{\eta}\right), \]
where $C$ depends on the parameters~$\kappa,~\mu_1,~\mu_2,~\mu_3,~b,~\beta$.
\end{theorem}

Again, as in the linear case, the choice of the regularization parameter is crucial. By balancing the first and second term, we obtain the parameter choice 
\begin{equation}
\label{eq:choice-non-lin}
\la_{m, \theta}=\paren{\frac{\Sigma}{\Ro\sqrt{m}}}^{\frac{2}{2r+b+1}}\;,
\end{equation}
with $\theta=(\Ro, \Sigma)$. A short calculation then shows that the third term is of lower order and that \eqref{l.la.condition.k} is satisfied for $m$ sufficiently large. We summarize our findings in the next corollary.

\begin{corollary}
\label{cor:non-lin-err.upper.bound}
Under the same assumptions of Theorem~\ref{convergence} with the parameter choice \eqref{eq:choice-non-lin} we have with probability at least $1-\eta$,
$$\|\fz-\fp\|_{\Ho}\leq C \Ro\paren{\frac{\Sigma}{\Ro\sqrt{m}}}^{\frac{2r}{2r+b+1}}\log\left(\frac{4}{\eta}\right)$$
where $C$ depends on the parameters~$\kappa,~\mu_1,~\mu_2,~\mu_3,~b,~\beta$ and provided that $m$ is sufficiently large.
\end{corollary}

Finally, proper integration using Corollary C.2 in \cite{Blanchard18} gives the upper rate of convergence.

\begin{theorem}
\label{theo:non-lin.rates.gen.u}
Suppose the assumptions of Corollary \ref{cor:non-lin-err.upper.bound} hold true. 
Then, the sequence 
\begin{align}
\label{eq:non-lin.rate.seq}
a_{m,\theta} = \Ro\paren{\frac{\Sigma}{\Ro\sqrt{m}}}^{\frac{2r}{2r+b+1}}
\end{align} 
is an upper rate of convergence in $L^p$ for all $p > 0$, for the reconstruction error, over the family of models $\cM(r, \Ro,  \cP_{\text{poly}}^{<} (\beta, b) )$ for the sequence of solutions \eqref{eq:Tikhonov-standard} with regularization parameter \eqref{eq:choice-non-lin}.
\end{theorem}

\begin{theorem}
\label{non-lin.rates.gen.l}
Let Assumptions~\ref{Ass:fp}--\ref{Ass:Kernel},~\ref{Ass:eigen.decay}, \ref{Ass:A.oper}--\ref{Ass:lower} hold true. 
%where $r > 0$, $\Ro > 0$, $b > 1$ and $\al > 0$ be fixed. 
Then the sequence $(a_{m,\theta})_\theta$ defined in~\eqref{eq:non-lin.rate.seq} is a weak lower rate of convergence in $L^p$ for all $p > 0$ for the reconstruction error, over the family of models $\cM(r, \Ro, \cP_{\text{poly}}^{>} (\alpha, b) )$. 
\end{theorem}

Finally, by Definition \ref{def:optimality}, we may conclude: 

\begin{theorem}
\label{theo:non-lin-rates-optimal}
Suppose the Assumptions of Theorem \ref{theo:non-lin.rates.gen.u} and Theorem \ref{non-lin.rates.gen.l} are satisfied. Then the sequence of solutions $(\fz)_m$  using the regularization parameters given in \eqref{eq:choice-non-lin} is weak minimax optimal in $L^p$, for all $p>0$, for the reconstruction error, over the model family $\cM(r, \Ro, \cP' )$, with $\cP'= \cP_{\text{poly}}^{<} (\beta, b)  \cap \cP_{\text{poly}}^{>} (\alpha, b) )$.
\end{theorem}

%Remarkably, this upper rate matches the minimax optimal rate for linear inverse problems from Section \ref{subsec:linear-classical}. This is a consequence of Assumptions \ref{Ass:A.oper} and  \ref{Ass:A.oper-2} we impose about the \emph{severity} of the non-linearity of $A$: The non-linear operator $A$ is \emph{sufficiently linearizable}.

%%%%%%%%%%%%%%%%%%%%%%%%%%%%%%%%%%%%%%%%%%%%%%%%%%%%%%%%%%%%%%%%%

\subsection{Error Bounds in Hilbert Scales}
\label{sec:non-lin-2}

Let $L: \cD(L) \subset \cH \to \cH$ be an unbounded operator, generating a Hilbert scale. 
The Tikhonov regularization scheme in Hilbert Scales amounts to minimize the regularized empirical risk 
\begin{equation}\label{eq:Tikhonov-standard.HS}
  f_{\zz,\la} =  \argmin_{f\in\mathcal{D}(A)\cap\dl}
  \left\{\frac{1}{m}\sum\limits_{i=1}^m\paren{[A(f)](x_i)-y_i}^2+\la \norm{L(f-\fbar)}_{\Ho}^2\right\},  
\end{equation}
enforcing smoothness and with initial guess $\bar f \in \cD(A) \cap \cD(L)$. As above, due to the non-linearity of $A$, the solution to the minimization problem is not explicitly given and does not need to be unique. To establish error bounds also in this case, the nonlinear structure of the operator needs to be further specified.

\begin{assumption}[Non-linearity of the operator III]
\label{Ass:A.oper.L}
We suppose that:
\begin{enumerate}[(i)]
\item The operator $A:\mathcal{D}(A)\cap \dl \to \Ht$ is weakly sequentially closed.

\item Link condition: There exist constants~$p\geq 1$ and~$0 < c_1 \leq c_2$ such that for all~$g\in \Ho$,
   \begin{equation*}
  c_1 \|g\|_{{-(p-1)}}\leq \|A'(\fp)g\|_{\nu}\leq c_2 \|g\|_{{-(p-1)}}.
   \end{equation*}
   
\item The Fr\'{e}chet derivative~$A'(f)$ is Lipschitz continuous for all~$f\in \mathcal{D}(A)\cap\dl~$, i.e., there exists a constant~$\mu_4$ such that
~$$\|A'(\fp)-A'(f)\|_{-(p-1),\nu}\leq \mu_4\|\fp-f\|_{\Ho}\leq \frac{c_1^2}{2c_2},$$
\end{enumerate}
where the norm $\norm{\cdot}_{-(p-1),\nu}$ is defined for the operators $\HH_{-(p-1)}$ to $\LL$.
\end{assumption}

The link condition, described in Assumption~\ref{Ass:A.oper.L} (ii)  involves the interplay between the operator $L^{-1}$ and the Fr{\'e}chet derivative of $A$.

\vspace{0.2cm}

We incorporate a-priori smoothness in terms of a more general source condition.  

\begin{assumption}[Source condition]
\label{ass:source-non-lin2}
Let $\bar f \in \cD(L)$, $s \in [1, p+1]$ for some benchmark smoothness $p\geq 1$. We assume that $\fp - \bar f \in \Omega(s, \Ro , L^{-1})$, where the source set is defined in Assumption \ref{Ass:source.HS}.   
\end{assumption}

Having now established all assumptions, we are ready to state an upper bound for the reconstruction error.

\begin{theorem}[Rastogi et al.~\cite{Rastogi20a}]\label{err.upper.bound}
Suppose Assumptions~\ref{Ass:fp}--\ref{Ass:Kernel}, \ref{Ass:eigen.decay}, \ref{Ass:A.oper}, \ref{Ass:A.oper.L} and \ref{ass:source-non-lin2} hold true. Then, for the Tikhonov estimator~$\fz~$ in~\eqref{eq:Tikhonov-standard.HS} with the a-priori choice of the regularization parameter~$\la=\paren{\frac{1}{\sqrt{m}}}^{\frac{2p}{p+s-1+bp}}$, for all~$0<\eta<1$, the following error bound holds with the confidence~$1-\eta$:
% $$\|\fz -\fp\|_{\Ho}=\mathcal{O}\paren{\la^r\log^2\paren{\frac{4}{\eta}}}\qquad {for}\quad r=\frac{s}{2(p+1)}.$$
$$\norm{\fz -\fp}_\Ho \leq C\paren{\frac{1}{\sqrt{m}}}^{\frac{s}{p+s-1+bp}}\log^2\paren{\frac{4}{\eta}}.$$
\end{theorem}

We will discuss the optimality of convergence rates later in this section.

\vspace{0.3cm}

{\bf Conditional stability estimates.} 
The above results require the differentiability of the operator~$A$. This condition can be relaxed by imposing \emph{conditional stability}. A  conditional stability estimate provides a tool to characterize the degree of ill-posedness in inverse problems and measures the  sensitivity of a solution with respect to changes in the input or observational data. This assumption implicitly addresses both the nonlinearity conditions and solution smoothness within the underlying nonlinear inverse problem. We will also establish the connection between link condition~\ref{Ass:A.oper.L} (ii) and the conditional stability estimates~\ref{Ass:A} (iv).

\vspace{0.2cm}

\begin{assumption}
\label{Ass:A}
\begin{enumerate}[(i)]

\item The domain~$\da$ of~$A$ is a convex and closed subset of~$\HH$.

\item The operator~$A : \da \to \HH'$ is weak-to-weak sequentially continuous\footnote{i.e.,~$f_n \rightharpoonup \hat{f}\in\HH$ with~$f_n \in \da$,~$n \in \NN$, and~$\hat{f} \in \da$ implies~$A(f_n)\rightharpoonup  A(\hat{f}) \in \HH'$.}.

\item The operator~$A$ is Lipschitz continuous with a Lipschitz constant~$\ell_A < \infty$ in a sufficiently large ball~$\mathcal{B}_d(\fp)$,
\begin{equation*}\label{A.cont}
\norm{A(f)-A(\tilde{f})}_{\HH'}\leq \ell_A \norm{f-\tilde{f}}_{\HH} \qquad\forall f ,\tilde{f}\in \mathcal{B}_d(\fp) \cap \mathcal{D}(A) \subset \HH,
\end{equation*}

\item Conditional stability estimate: There exist constants~$p\geq 1$,~$\al> 0$,~$d > 0$,~$\theta\geq 0$ and~$Q \subset\mathcal{D}_d^\theta(\fp)\cap \da$ such that
\begin{equation*}
\norm{f-\fp}_{-(p-1)}\leq \al\norm{A(f)-A(\fp)}_{\nu}
\end{equation*}
for all~$f \in Q$, where the constant~$\al$ may depend on~$p$, and~$Q$.
\end{enumerate}
\end{assumption}

\vspace{0.2cm}

Conditional stability estimates require additional regularization for obtaining stable approximate solutions if the validity area $Q$ of such estimates is not completely known. Conditional stability estimates can (in the absence of differentiability) be verified by means of global inequalities of the forward operator $A$. For examples, we refer the reader to \cite{werner2019convergence}.

\vspace{0.2cm}

%It's important to highlight that the operator~$A$ may not be differentiable in this context.

\begin{theorem}[Rastogi et al.~\cite{Rastogi22}]\label{err.upper.bound.p}
Suppose Assumptions~\ref{Ass:fp}--\ref{Ass:Kernel}, \ref{Ass:eigen.decay},~\ref{ass:source-non-lin2} and~\ref{Ass:A}  hold true and~$\fz,\fp\in Q$ (for a sufficiently large sample size~$m$) for~$p$,~$Q$, and~$s$ defined in Assumption~\ref{Ass:A}~(iv). Then, for the Tikhonov estimator~$\fz~$ in~\eqref{eq:Tikhonov-standard.HS} with the a-priori choice of the regularization parameter~$\la=\paren{\frac{1}{\sqrt{m}}}^{\frac{2p}{p+s-1+bp}}$, for all~$0<\eta<1$, the following error bound holds with the confidence~$1-\eta$:
% $$\|\fz -\fp\|_{\Ho}=\mathcal{O}\paren{\la^r\log^2\paren{\frac{4}{\eta}}}\qquad {for}\quad r=\frac{s}{2(p+1)}.$$
$$\norm{\fz -\fp}_\Ho \leq C\paren{\frac{1}{\sqrt{m}}}^{\frac{s}{p+s-1+bp}}\log^2\paren{\frac{4}{\eta}}.$$
\end{theorem}

\vspace{0.3cm}

{\bf Optimality.} 
Now, we discuss the optimality of the convergence rates established in this section. To do so, we rely on the lower bounds for the linear inverse problem presented in~\cite{Rastogi23}, which pertain to a similar setting, albeit with a link condition.

First, we establish a connection between Assumption~\ref{Ass:A} (iv) and the link conditions~\ref{Ass:A.oper.L} (ii),~\ref{Ass:link}. The conjunction of Assumptions~\ref{Ass:A} (iv) and
\begin{equation*}
 \norm{A(f)-A(\fp)}_{\nu} \leq \beta\norm{f-\fp}_{\HH_{-(p-1)}}, \qquad \beta>0
\end{equation*}
leads to the following implication:
\begin{equation}\label{link.con}
   \norm{f-\fp}_{\HH_{-(p-1)}} \asymp  \norm{A(f)-A(\fp)}_{\nu}. 
\end{equation}
 Here, we write $\norm{B_1(f)}\asymp\norm{B_2(f)}$ if there exist $c_1 \leq c_2 < \infty$, such that 
    $c_1\norm{B_1(f)}\leq  \norm{B_2(f)} \leq c_2 \norm{B_1(f)}$.

Let us next introduce an additional assumption that is utilized in both \cite{Hanke95} and \cite{kaltenbacher08}:
\begin{assumption}\label{Ass:add}
For $f,\fp\in\da$,
  \[\norm{A(f)-A(\fp)}_{\nu} \asymp \norm{ A'(\fp)\paren{f-\fp}}_{\nu}.  \] 
\end{assumption}

The combination of Assumption~\ref{Ass:add} and equation~\eqref{link.con} leads to link condition~\ref{Ass:A.oper.L} (ii):
$$\norm{f-\fp}_{\HH_{-(p-1)}} \asymp \norm{A'(\fp)\paren{f-\fp}}_{\nu}.$$

Next, by substituting $f-\fp=L^{-1}h$, we derive:
$$\norm{L^{-1}h}_{\HH_{-(p-1)}} \asymp \norm{ A'(\fp)L^{-1}h}_{\nu}$$
which further implies:
$$\norm{h}_{\HH_{-p}} \asymp \norm{(L^{-1}\tp L^{-1})^{\frac{1}{2}} h}_{\nu},\quad \text{for} \quad \tp=(\ip A'(\fp))^*(\ip A'(\fp)).$$

The link condition~\ref{Ass:link} considered in the linear case coincides with the above inequality for $p=\frac{1}{2a}$.

Hence, through a comparison of the aforementioned link relationships, we deduce that we need to assign~$\frac{1}{2a}\leftarrow p$.

As reported in~\cite{Rastogi23} for the statistical inverse problem, assuming smoothness~$s$ in the source condition~\ref{Ass:source.HS}, and considering eigenvalue decay condition~\ref{Ass:eigen.decay} with parameter $b$, the optimal rate is of the order~$\paren{\frac{1}{\sqrt{m}}}^{\frac{s}{s-1 + \frac{1}{2a} +\frac{b}{2a}}}$. In this context, we have to assign~$\frac{1}{2a}\leftarrow p$. This yields a lower bound of the order
$$
 \paren{\frac{1}{\sqrt{m}}}^{\frac{s}{p+s-1+bp}}.
$$
This corresponds to the upper bound discussed in Theorems~\ref{err.upper.bound}, \ref{err.upper.bound.p}, demonstrating that the rates are of optimal order.

%%%%%%%%%%%%%%%%%%%%%%%%%%%%%%%%%%%%%%%%%%%%%%%%%%%%%%%%%%%%%%%%%%%%%%%%%%%%%%%%%%%%%

\section{Applications}\label{Sec:Applicaions}

%\subsection{Tomography}\label{Sec:tomography}

%\subsection{(Nonlinear application)}\label{Sec:Pharmacology.exam}

Pharmacokinetic/pharmacodynamic (PK/PD) models are employed to predict the evolution of drug concentrations and effects in patients undergoing treatment. Covariate models tell us how patient attributes (covariates) influence model parameters and play a crucial role in PK/PD model construction. Common covariates are body weight, age, and levels of crucial biomarkers (e.g., plasma creatinine). The problem of interest is the determination of the functional relationship between covariates and parameters.  

\vspace{0.3cm}

{\bf Mechanistic model $G$.} 

In the two-compartment Pharmacokinetic model~\cite{Pilari10}, we consider the human body into two compartments: the central and peripheral compartments. It can be described by the system of ODEs:
\begin{align*}
  V_1\frac{dC_1}{dt}= & Q (C_2 - C_1)- \text{CL} \cdot C_1,\\
  V_2\frac{dC_2}{dt}= & Q (C_1 - C_2).
\end{align*}

The central and peripheral compartments contain drug concentrations represented by $C_1$ and $C_2$ in mg/L, respectively. These compartments have volumes $V_1$ and $V_2$ in L, and are connected by intercompartmental flow $Q$ in L/day, with clearance $\text{CL}$ in L/day. Initial conditions are set as $C_1 (0) = D_{\text{rel}} \cdot \frac{w}{V_1}$ and $C_2 (0) = 0$, where $D_{\text{rel}}$ is an intravenous bolus dose of 15 mg/kg body weight. The covariate-to-parameter relationship includes a maturation component ($\text{mat}(a)$) dependent on age $a$ and an allometric component dependent on weight $w$.
\begin{eqnarray*}
&V_1 =V_1^* \paren{\frac{w}{w_{\text{ref}}}}; &V_2 =V_2^* \paren{\frac{w}{w_{\text{ref}}}};\\
&Q=Q^* \paren{\frac{w}{w_{\text{ref}}}}^{\frac{3}{4}}; &\text{CL} =\text{CL}^* \cdot\text{mat}(a) \paren{\frac{w}{w_{\text{ref}}}}^{\frac{3}{4}}; 
\end{eqnarray*}
with reference body weight $w_{\text{ref}} = 70$ kg. The weight-normalized parameters $\paren{V_1^*, V_2^*, Q^*,\text{CL}^*}$ are unknown.

The model's parameters are denoted by $\theta = \paren{ V_1, V_2, Q, \text{CL}}$. The observed quantity
$$G(x,\theta) =\paren{\text{ln}~C_1(t_1),\ldots, \text{ln}~C_1(t_q)}$$
was measured at fixed time points $t_1,\ldots, t_q$. 

Hence, this problem can be described as a statistical inverse problem in~\cite{Hartung21}:
$$y_i = G(x_i,\theta_i)+\varepsilon_i, \qquad \theta_i=f(x_i),~~~~i\in \brac{1,\ldots, m},$$
where the observations consist of covariates (age $a$, weight $w$)~$x_i\in X \subset  \RR^2$ and (noisy) drug concentration measurements $y_i \in \RR^q$. There are unobserved parameters~$\theta_i\in \Theta\subset\RR^4$, along with a mechanistic model $G:X \times \Theta \to \RR^q$, which is a nonlinear function. The mapping~$f$ takes covariates to parameters. Moreover, the noise~$\varepsilon_i$ is assumed to be normally distributed centered noise.  

The mechanistic model~$G$ represents a solution to a system of ordinary differential equations (ODEs) governed by a two-compartment model observed at different time points $t_1,\ldots, t_q$. The direct dependence of $G$ on $x_i$ allows us to model personalized doses and focus on specific aspects of a covariate model. The operator~$G$ exhibits a nonlinear dependence on the parameters. Here, the aim is to estimate the unknown covariate model $f$.

Hartung et al.~\cite{Hartung21} used Tikhonov regularization to estimate this covariate model in RKHS. The kernel is chosen empirically. Frequently, a specific parametric category of covariate-to-parameter relationships (e.g., linear, exponential, etc.) is chosen on an ad hoc basis. Therefore, the authors conducted hypothesis tests to check the structure of the covariate model $f$, i.e., whether it belongs to this parametric class or nonparametric class as a kernelized Tikhonov regularized. Then, the authors assessed the effect of age on the clearance of a monoclonal antibody.

Now, we check the validity of the assumptions considered for Theorem~\ref{theo:non-lin.rates.gen.u} for the upper rates of convergence. First, we focus on the nonlinear structure assumptions~\ref{Ass:A.oper} and~\ref{Ass:A.oper-2} for the non-linear operator~$A(f)(x)=G(x,f(x))$. Let us record the following theorem from \cite{Valent87} regarding its properties.
\begin{theorem}[{\cite[Thm.~3.1]{Valent87}}]
Let us assume that~$\Theta$ is a nonempty, open, bounded subset of~$\RR^p$ with suitably smooth boundary and $X\subset \RR$. For any $f\in \HH=W^{k,2}(\RR)$ (see \cite[Sec.~1.3.5]{Saitoh}) we have
~$G\in \mathcal{C}^{k+1}(\RR \times \overline{\Theta})$.
Moreover, we observe that the operator~$A : \Ho \to \Ho$ is continuous and the Fr{\'e}chet derivative of~$A$ at~$f$ is given by
  \begin{equation*}
    [A'(f)g](x)=g(x) D_\theta G(x,f(x)),
  \end{equation*}
where~$D_\theta$ denotes the derivative of~$G$ with respect to the second coordinate.   
In addition, the Fr{\'e}chet derivative is Lipschitz continuous.
\end{theorem}

Since the derivative of $G$ is Lipschitz continuous in the first coordinate and $\Theta$ is bounded. Therefore,
$$\abs{[A'(f)g](x)}\leq \abs{D_\theta G(x,f(x))} \abs{g(x)} \leq  c_G \norm{x} \abs{\inner{K_x,g}} \leq \kappa  c_G \norm{g}\norm{x}$$
implies that
$$\norm{A'(f)}\leq \kappa c_G.$$ 

Hence, the Fr{\'e}chet derivative is bounded. The Lipschitz continuity of the Fr{\'e}chet derivative of $A$ implies the condition~\ref{Ass:A.oper-2} (ii) (see \cite[Chapt.~10]{Engl96}). 

The normally distributed centered noise satisfies the assumptions~\ref{Ass:fp},~\ref{Ass:noise}. The Sobolev space $W^{k,2}(\RR)$ satisfies the assumption~\ref{Ass:Kernel}. In case the design measure satisfies the source condition~\ref{ass:source-non-lin} and the eigenvalue decay condition~\ref{Ass:eigen.decay} (i), we obtain the minimax rates of convergence for the parameter choice provided in Theorem~\ref{theo:non-lin.rates.gen.u}.

%%%%%%%%%%%%%%%%%%%%%%%%%%%%%%%%%%%%%%%%%%%%%%%%%%%%%%%%%%%%%%%%%%%%%%%%%%%%%
%%%%%%%%%%%%%%%%%%%%% Conclusion and Outlook
%%%%%%%%%%%%%%%%%%%%%%%%%%%%%%%%%%%%%%%%%%%%%%%%%%%%%%%%%%%%%%%%%%%%%%%%%%%%%

\section{Conclusions and outlook  for future directions }
\label{Sec:outlook}

Motivated by the growing interest in supervised inverse learning with random design, we hope to have presented a thorough introduction to this modern research field. To analyze both linear and non-linear inverse problems, we have defined a general mathematical framework within reproducing kernel Hilbert spaces,  allowing us to appropriately consider the stochastic nature of the data.

\vspace{0.2cm}

For tackling linear inverse problems, we adopt a least squares approach with various types of regularization: spectral regularization, regularization by projection, and convex regularization. We discuss optimal rates of convergence for our estimators, initially dependent on unknown structural model assumptions such as the smoothness of the ground truth and the eigenvalue decay of the uncentered covariance operator (both describing the degree of ill-posedness of the inverse problem), ultimately leading to fast rates of convergence. To prevent overfitting, a proper choice of the regularization parameter is crucial. Since a-priori choices are not feasible in practical applications, we gave a brief overview of common data-driven approaches.

\vspace{0.2cm}

While inverse learning with a known linear forward operator is well developed for a fair amount of regularization methods, much less is known for inverse learning with a known forward non-linear operator. We discussed error bounds only for Tikhonov regularization in both the classical setting and within Hilbert scales using the least squares loss. 
A Standard approach for handling non-linear problems is to consider linear approximations of the operator: Under the assumption of Fr\'{e}chet differentiability of $A$ and some Lipschitz conditions, the non-linearity is mild enough to obtain optimal bounds, matching those of the linear case.  Notably, in the case where $A$ is not  Fr\'{e}chet differentiable, error bounds can still be derived by means of conditional stability estimates.

\vspace{0.2cm}

Unfortunaltely, in the non-linear case, the Tikhonov regularized solution is not directly computable, and hence, other regularization approaches that take the non-linear structure into account and are efficiently implementable need to be analyzed further.  Due to their practicability, the hope is to successfully apply kernel methods also to obtain approximate solutions that maintain statistical optimality. A major focus should be on the development of data-driven approaches that are also efficient, such as early stopping rules. We hope that our contribution can provide an impetus to further research in this area.

%For instance, gradient based methods are useful only  under suitable assumptions on the Fr\'{e}chet 
%derivative. Hopefully, kernel methods can be applied as well to find fully data-driven approximate minimizers, enjoying  optimal adaptive statistical properties. 

%%%%%%%%%%%%%%%%%%%%%%%%%%%%%%%%%%%%%%%%%%%%%%%%%%%%%%%%%%%%%%%%%%%%%%%%%%%%%
%%%%%%%%%%%%%%%%%%%%% List of Symbols 
%%%%%%%%%%%%%%%%%%%%%%%%%%%%%%%%%%%%%%%%%%%%%%%%%%%%%%%%%%%%%%%%%%%%%%%%%%%%%

%\appendix 

\addcontentsline{toc}{section}{List of symbols and notation}

\vspace{1cm}

{\bf \Large List of symbols and notation}
%LIST OF SYMBOLS AND NOTATION:

\vspace{0.5cm}

{\bf Samples and sample space:}
\begin{itemize}
\item $X=\RR^d$ input space, $Y=\RR$ output space, $Z=X\times Y$ sample space
\item $\xx=(x_1,\ldots,x_m)$ input data, $\yy=(y_1,\ldots,y_m)$ output data and $\zz=((x_1,y_1),\ldots,(x_m,y_m))$ full training data
\item $m$ sample size
\end{itemize}

\vspace{0.3cm}

{\bf Spaces and norms:}
\begin{itemize}
    \item $\Ho$, $\Ht$ Hilbert spaces with inner products $\inner{\cdot,\cdot}_{\Ho}$ and $\inner{\cdot,\cdot}_{\Ht}$, respectively
    \item  $(\HH_a)_{a \in \mbr}$, Hilbert scale with the inner product $\inner{\cdot,\cdot}_a$ and norm $||\cdot||_a$ 
    \item  $||\cdot||$ operator norm
    \item  $||\cdot||_2$ euclidean norm
    \item  $\LLL^p(X, \nu)$  space of $p$-integrable functions on $X$, the $\LL$-norm is denoted by $||\cdot||_\nu$
\end{itemize}

\vspace{0.3cm}

{\bf Operators:}
\begin{itemize}
\item $\mathcal{D}(A)$ domain of an operator $A$, $\range(A)$ range of an operator $A$
    \item $A: \mathcal{D}(A) \subset\cH \to \cH'$ linear/nonlinear forward operator
     \item $A'(f)$ Fr{\'e}chet derivative of $A$ at $f$
    \item  $L$ unbounded operator generating a Hilbert scale
    \item $\ip : \HH' \hookrightarrow \LL$ canonical injection operator  
    \item  $\bp=\ip \circ A'(f)$,  $\tp=\bp^*\bp$ uncentered covariance operator
    \item  $\sx$ sampling operator, $\bx= \sx \circ A'(f)$, $\tx=\bx^*\bx$ 
    \item $P_n$ orthogonal projection onto $V_n$, an  $n$-dimensional subspace of $\Ho$
    \item  $B^\dagger$ pseudoinverse of an operator $B$  
   
\end{itemize}

\vspace{0.3cm}

{\bf  Measures:}
\begin{itemize}
  \item $\rho$ is the distribution of $(x,y)$ 
  \item $\nu$ marginal distribution on $X$
  \item $\cP(X)$ class of marginal distribution on Borel $\sigma$-field $\cB(X)$
  \item $\cP_{\text{poly}}^{<} (\beta, b)$, $\cP_{\text{poly}}^{>}(\alpha , b)$, $\cP_{\text{str}}^{>}(\gamma)$ sub-classes of $\cP(X)$
  \item  $\mathcal{M}_\theta$ model classes
\end{itemize}

\vspace{0.3cm}

{\bf Functions:}
\begin{itemize}
 \item $\fp$, $\gp$: ground truths with $\dagger$ symbol
    \item $\fbar$ initial guess of the ground truth
     \item    $\ga$ spectral regularization
    \item   $\fhat$ estimated solution based on samples
    \item   $\fz$ regularized solution 
    \item   $d_p$ distance function 
\end{itemize}

\vspace{0.3cm}

{\bf Parameter: }
\begin{itemize}
 \item $\varepsilon$ observation noise
  \item  $\kappa^2$ upper bound for the reproducing kernel $K$ on $X$
   \item $\sigma_i$, $\sigma_{\min}$ eigenvalues/spectral decay etc 
    \item $t$ spectral decay of $\tp$
    \item $\la$  the regularization parameter
    \item    $b$ eigenvalue decay parameter
    \item $a$ and $p$ for link condition
    \item  $a$, $p$ parameters in link condition and $p$ is also the benchmark smoothness in distance function
    \item $q$ qualification of the regularization
    \item $\Sigma$, $M$ parameters in the noise condition~\ref{Ass:noise}
     \item $\Ro$ norm bound% ($R$ is used for distance function)
    \item $s,p,r$ all relevant to smoothness in source conditions
\end{itemize}

%%%%%%%%%%%%%%%%%%%%%%%%%%%%%%%%%%%%%%%%%%%%%%%%%%%%%%%%%%%%%%%%%%%%%%%%%%%%%
%%%%%%%%%%%%%%%%%%%%% References
%%%%%%%%%%%%%%%%%%%%%%%%%%%%%%%%%%%%%%%%%%%%%%%%%%%%%%%%%%%%%%%%%%%%%%%%%%%%%

\bibliographystyle{plain}
\bibliography{references}

\end{document}